\documentclass[lettersize,journal]{IEEEtran}
\usepackage{algorithm,algorithmic, cite}
\usepackage{amsmath,amssymb,amsfonts}
\usepackage{graphicx}
\usepackage{textcomp}
\usepackage{xcolor, xr, xcite}
\usepackage{hyperref}
\usepackage{amsbsy}
\usepackage{amsthm}
\usepackage{subcaption}
\usepackage{enumerate, enumitem}
\usepackage{float}
\usepackage{cleveref}

\usepackage{ragged2e} 
\usepackage{tabularx} 
\usepackage{multirow}
\usepackage{diagbox} 
\usepackage{array}   
\usepackage{booktabs} 
\usepackage{makecell}   
\usepackage{siunitx}    

\usepackage{bm}
\providecommand{\abs}[1]{\lvert#1\rvert}
\renewcommand{\b}[1]{\ensuremath{\mathbf{#1}}} 
\newcommand{\Ex}[1]{\ensuremath{\EE[#1]}}  
\newcommand{\Et}[1]{\ensuremath{\mathbb{E}_t[#1]}}  
\newcommand{\norm}[1]{\ensuremath{\left\|#1\right\|}} 
\newcommand{\mat}[1]{\ensuremath{\begin{bmatrix}#1\end{bmatrix}}} 
\providecommand{\tr}[1]{\mathrm{tr}\left(#1\right)} 
\newcommand{\eqtext}[1]{\ensuremath{\stackrel{\text{#1}}{=}}} 
\newcommand{\leqtext}[1]{\ensuremath{\stackrel{\text{#1}}{\leq}}} 
\providecommand{\ip}[2]{\langle #1, #2 \rangle} 
\newcommand{\col}[1]{\textcolor{blue}{#1}} 

\renewcommand{\O}[1]{\mathcal{O}}

\def \e {{\b{e}}}

\def \g {{\b{g}}}
\def \h {{\b{h}}}
\def \n {{\b{n}}}
\def \q {{\b{q}}}
\def \p {{\b{p}}}
\def \r {{\b{r}}}
\def \s {{\b{s}}}
\def \u {{\b{u}}}
\def \v {{\b{v}}}
\def \w {{\b{w}}}
\def \x {{\b{x}}}
\def \y {{\b{y}}}
\def \z {{\b{z}}}
\def \zero {{\mathbf{0}}}

\def \grd {{\text{grad}}}

\def \EE {{\mathbb{E}}}

\def \G {{\mathbf{G}}}

\def \cC {{\mathcal{C}}}
\def \cF {{\mathcal{F}}}
\def \cG {{\mathcal{G}}}
\def \O {{\mathcal{O}}}

\def \cL {{\mathcal{L}}}

\def \cI {{\mathcal{I}}}
\def \cN {{\mathcal{N}}}

\def \cT {{\mathcal{T}}}

\def \A {{\mathbf{A}}}
\def \B {{\mathbf{B}}}
\def \C {{\mathbf{C}}}
\def \D {{\mathbf{D}}}

\def \G {{\mathbf{G}}}
\def \H {{\mathbf{H}}}
\def \I {{\mathbf{I}}}

\def \M {{\mathbf{M}}}
\def \N {{\mathbf{N}}}
\def \P {{\mathbf{P}}}
\def \Q {{\mathbf{Q}}}
\def \R {{\mathbf{R}}}

\def \U {{\mathbf{U}}}
\def \V {{\mathbf{V}}}
\def \W {{\mathbf{W}}}
\def \X {{\mathbf{X}}}
\def \Y {{\mathbf{Y}}}

\def \sZ {{\mathsf{Z}}}

\providecommand{\ipx}[2]{\langle #1, #2 \rangle_{\X}} 

\def \xib {{\boldsymbol{\xi}}}

\def \etab {{\boldsymbol{\eta}}}

\def \T {{\mathsf{T}}}

\def \Rn {{\mathbb{R}}}
\def \Pn {{\mathbb{P}^n}}
\def \Sn {{\mathbb{S}^n}}

\def \Expx {{\mathrm{Exp}_\X}}
\def \Expxt {{\mathrm{Exp}_{\X_t}}}

\makeatletter
\AtEndDocument{%
	\begingroup
	\edef\@currentlabel{\number\value{equation}}\label{ctr:last-equation}%
	\edef\@currentlabel{\number\value{lemma}}\label{ctr:last-lemma}%
	\endgroup
}
\makeatother

\newtheorem{assumption}{}

\newtheorem{theorem}{Theorem}
\newtheorem{lemma}{Lemma}
\newtheorem{definition}{Definition}

\allowdisplaybreaks

\date{}
\usepackage{lastpage}

\begin{document}

\title{Rank-one Riemannian Subspace Descent for Nonlinear Matrix Equations
 \author{Yogesh Darmwal, Ketan Rajawat}}
\maketitle

\begin{abstract}
We propose a rank-one Riemannian subspace descent algorithm for computing symmetric positive definite (SPD) solutions to nonlinear matrix equations arising in control theory, dynamic programming, and stochastic filtering. For solution matrices of size $n\times n$, standard approaches for dense matrix equations typically incur $\O(n^3)$ cost per-iteration, while the efficient $\O(n^2)$ methods either rely on sparsity or low-rank solutions, or have iteration counts that scale poorly. The proposed method entails updating along the dominant eigen-component of a transformed Riemannian gradient, identified using at most $\O(\log(n))$ power iterations. The update structure also enables exact step-size selection in many cases at minimal additional cost. For objectives defined as compositions of standard matrix operations, each iteration can be implemented using only matrix--vector products, yielding $\O(n^2)$ arithmetic cost. We prove an $\O(n)$ iteration bound under standard smoothness assumptions, with improved bounds under geodesic strong convexity. Numerical experiments on large-scale CARE, DARE, and other nonlinear matrix equations show that the proposed algorithm solves instances (up to $n=10{,}000$ in our tests) for which the compared solvers, including MATLAB's \texttt{icare},  structure-preserving doubling, and subspace-descent baselines fail to return a solution. These results demonstrate that rank-one manifold updates provide a practical approach for high-dimensional and dense SPD-constrained matrix equations. MATLAB code implementation is publicly available on GitHub : \href{https://github.com/yogeshd-iitk/nonlinear_matrix_equation_R1RSD}{\col{https://github.com/yogeshd-iitk/nonlinear\_matrix \_equation\_R1RSD}}
\end{abstract}

\section{Introduction}\label{intro}
Symmetric positive definite (SPD) solutions of algebraic Riccati and Lyapunov equations are fundamental in modern control, providing standard certificates and tools for analyzing stability, performance, and robustness \cite{zhou1996robust}.  However, as the state dimension grows, computing stabilizing SPD solutions reliably and efficiently becomes a major computational bottleneck. In the absence of exploitable structure, such as sparsity or low rank, most state-of-the-art solvers require $\O(n^3)$ or more floating point operations (flops) at every iteration when the state dimension is $n$ \cite{benner2004solving, benner2020numerical}. 


In this work, we consider nonlinear matrix equations (NMEs) involving a matrix variable $\X$ and constant matrices $\A,\C,\ldots$, of the form
\begin{align}
\cG\left(\X,\A,\C, \ldots\right)&=	\zero & \X \in \Pn, \tag{$\mathcal{P}_e$} \label{eqn_form}
\end{align}
where $\Pn$ denotes the set of $n \times n$ real symmetric positive definite (SPD) matrices. Here, $\cG$ is a matrix-valued nonlinear mapping built from standard algebraic matrix operations, such as transpose, addition, multiplication, inverse, and log-determinant. We recast \eqref{eqn_form} as a nonlinear least-squares problem
\begin{align}
	\min_{\X \in \Pn} \norm{\cG\left(\X,\A,\C, \ldots\right)}_F^2\label{residual_norm}
\end{align}
which enables Riemannian optimization directly on $\Pn$ and paves the way for the proposed computationally efficient solver. Moreover, any SPD solution to \eqref{eqn_form} is a global minimizer of \eqref{residual_norm} with objective value zero. If \eqref{eqn_form} is not solvable due to errors or perturbations in the coefficient matrices, \eqref{residual_norm} still returns a meaningful best fit SPD solution. 

Specific instances of \eqref{eqn_form} arising in control theory include continuous- and discrete-time algebraic Lyapunov equations (CALE and DALE), continuous- and discrete-time algebraic Riccati equations (CARE and DARE) \cite{anderson2007optimal}, generalized CARE and DARE \cite{ionescu1997general}, as well as the stochastic CARE and DARE \cite{rami2000linear}. Beyond control theory, instances of \eqref{eqn_form} also appear in nano research \cite{guo2012numerical, bollhofer2017low}, interpolation theory \cite{ran2004nonlinear}, and ladder networks \cite{anderson1983ladder}. 

Over the years, numerous algorithms for computing positive definite solutions have been proposed and analyzed. Examples include sign function methods \cite{benner2015matrix}, Alternating Direction Implicit (ADI) iterations \cite{wachspress2013adi}, fixed-point iterations \cite{dai2011eigenvalue, meng2016positive, meng2017positive, haqiri2017methods, meng2018positive}, structure-preserving doubling algorithms \cite{lin2006convergence, huang2018structure, huang2019some, zhang2020structure}, Newton’s and quasi-Newton methods \cite{huang2018some, weng2021solving}, multi-step stationary iterative methods \cite{zhai2021solvability}, gradient-based methods \cite{ding2005gradient, ding2006iterative, xie2016accelerated}, inverse-free methods \cite{monsalve2010new, huang2013inversion, huang2018some, wang2024new}, quasi-gradient-based inversion-free methods \cite{zhang2019quasi, huang2018some}, and accelerated algorithms \cite{lin2018accelerated, li2022fixed}. Despite this breadth, the per-iteration complexity of all these algorithms remains $\O(n^3)$ or higher for the general case, making them unsuitable for large-scale problems. Algorithms achieving lower per-iteration complexities typically rely on additional assumptions such as $\X$ being low rank and coefficient matrices being sparse \cite{li2002low, benner2009adi, vandereycken2010riemannian, benner2015matrix, benner2016inexact, benner2020numerical}. However, without such low-rank and sparsity assumptions, the per-iteration cost of these algorithms remains $\O(n^3)$, limiting their applicability to high-dimensional problems \cite{yu2019large}.

Riemannian geometric methods have gained significant attention \cite{lee2008invariant, jung2009solution, duan2013natural, han2021riemannian, han2024riemannian} for their ability to leverage the intrinsic geometry of SPD constraint. To reduce per-iteration costs, subspace/coordinate descent algorithms that update only a subset of directions \cite{bertsekas1999nonlinear, nesterov2012efficiency, fercoq2015accelerated} have recently been extended to manifold settings \cite{celledoni2008descent, shalit2014coordinate, gao2018new, gutman2023coordinate}. 
For the SPD manifold in particular, subspace descent algorithms were proposed in \cite{darmwal2023low, han2024riemannian}. Of these, \cite{han2024riemannian} is the only method that reduces the dominant per-iteration cost on the SPD manifold to $\O(n^2)$ flops for the general case of \eqref{residual_norm}. This quadratic scaling is a practical prerequisite in large-scale settings, since dense $O(n^3)$ operations can become prohibitive to the point that standard implementations may be unable to run for large $n$ due to time and memory constraints. For example, when $n=10^4$, a single dense $\mathcal{O}(n^3)$ iteration requires $10^{12}$ floating point operations, whereas a  $\mathcal{O}(n^2)$ iteration requires only $10^{8}$ operations, often marking the difference between infeasible and tractable computation. Although subspace descent approaches generally require more iterations, they remain viable at problem sizes where standard approaches are computationally infeasible.

In this work, we propose a scalable Rank-one Riemannian Subspace Descent (R1RSD) algorithm for solving \eqref{residual_norm}. The proposed algorithm features rank-one updates that require only $\O(n^2\log n)$ flops per-iteration and $\O(n)$ total iterations under standard assumptions. The proposed algorithm improves upon the state-of-the-art subspace descent method in \cite{han2024riemannian} which has a similar $\O(n^2)$ per-iteration complexity but requires $\O(n^2)$ total iterations.  We summarize our key contributions as follows.
\begin{itemize}[leftmargin=*]
	\item Different from various Riemannian subspace descent algorithms that are based on projecting the Riemannian gradient onto a subspace of the tangent space, we express a transformed version of the Riemannian gradient as a sum of rank-one eigen-components and then update the iterate along a carefully selected component at each step. 
	\item We employ the power method to identify the component that yields the largest first-order decrease within the candidate family and show that some objectives admit closed-form step-size selection at every iteration without significant additional effort. In many cases, the proposed algorithms also allow efficient calculation of the residual norm objective, which may otherwise be difficult \cite{kurschner2016efficient, benner2020numerical}, and can be used as a stopping criteria or to check global optimality. 
	\item We provide a detailed analysis of the proposed algorithm, including the inexact power method direction selection, establishing the $\O(n^2\log n)$ bound on the flops per-iteration and $\O(n)$ bound on the iteration complexity for both, the general non-convex case and the geodesically strongly convex case. 
	\item We benchmark the numerical performance of the proposed algorithm on large-scale CAREs, comparing the performance against state-of-the-art algorithms \cite{han2024riemannian,huang2018structure} as well as MATLAB's built-in function \texttt{icare}. We further report scaling experiments that identify the problem-size regime where these standard solvers become computationally prohibitive, while R1RSD remains practical due to its rank-one update structure. We also evaluate the performance on an instance of DARE and another commonly used NME. 
\end{itemize}
The proposed algorithm may be seen as updating along a $\O(n)$-dimensional subspace as was also proposed in the RRSD-multi algorithm of \cite{darmwal2023low}. However, the present work targets structured rank-one updates for a specific class of problems and hence achieves lower $\O(n^2)$ computational complexity per-iteration, as opposed to the $\O(n^3)$ complexity achieved by \cite{darmwal2023low} when applied to \eqref{residual_norm}. To support reproducibility and adoption, we provide an implementation and scripts that reproduce the main experiments reported in the paper.

Finally, we remark that the proposed low-complexity algorithms can also be applied to time-varying, adaptive, and nonlinear feedback control, where solving Riccati or Lyapunov equations at every time step is infeasible \cite{ioannou2006adaptive}. In particular, within the State-Dependent Riccati Equation (SDRE) framework for nonlinear control \cite{saluzzi2025state}, the proposed methods can enable efficient tracking of slowly varying Riccati solutions using warm starts and a small number of iterations per time step. Compared to the standard cascade Newton--Kleinman methods \cite{saluzzi2025dynamical}, the resulting approach would achieve almost $n$-fold reduction in the per time-step complexity, making it suitable for high-dimensional systems. This facilitates online and embedded workflows in which Riccati or Lyapunov equation must be solved repeatedly, e.g., within adaptive control loops or receding-horizon updates.

\subsection{Related Work}
Classical numerical methods for solving Lyapunov/Sylvester equations, generalized linear matrix equations, and Riccati equations, such as invariant or deflating subspace techniques, ADI iterations, fixed-point methods, generalized Schur-based solvers, and Kleinman–Newton variants, are surveyed in \cite{datta2004numerical,bini2011numerical,simoncini2016computational}. Over the past two decades, structure-preserving doubling algorithms have been extensively developed for Riccati equations and certain classes of NMEs \cite{huang2018structure}. Among these, all the iterative approaches incur a per-iteration computational cost of $\mathcal{O}(n^3)$. 

A comparative overview of state-of-the-art algorithms for solving general NMEs, in terms of per-iteration cost and convergence behavior, is provided in Table~\ref{sota_algo}. Invariant or deflating subspace techniques, fixed-point iterations, and structure-preserving doubling algorithms are primarily tailored to solve the Riccati equation and a limited class of NMEs. More general approaches, including Riemannian gradient-based algorithms, can be applied to solve \eqref{residual_norm} but typically exhibit a per-iteration cost of $\mathcal{O}(n^3)$. An exception is the Burer–Monteiro factorization-based (BMFC) approach, which reduces per-iteration complexity at the expense of operating in a higher-dimensional search space.

Beyond Lyapunov, Sylvester, and Riccati equations, several classes of NMEs have been studied in the literature. One such class and its generalizations are investigated in \cite{Anderson1990positive,huang2018some,weng2021solving,li2022fixed,erfanifar2022efficient,jia2023hermitian,engwerda1993necessary,ferrante1996hermitian,ivanov2005improved,zhou2013positive}, where Newton-type and fixed-point methods are proposed. Another class is considered in \cite{meng2017positive,mouhadjer2022new,li2014hermitian,yu2022iterative,jin2022nonlinear}, which primarily employs fixed-point and inverse-free fixed-point iterations. Further generalizations involving real powers of the matrix variable are studied in \cite{hasanov2005positive,yin2010positive,cai2010hermitian,zhang2011positive,peng2007positive,lim2009solving,fang2017positive,jin2021investigation,liu2011hermitian,liu2014hermitian,li2019investigation,yin2014nonlinear,erfanifar2022efficient}, again largely relying on inverse-free fixed-point schemes. Finally, a few works also address matrix equations involving transcendental matrix functions, such as the matrix exponential, and propose fixed-point algorithms for solving them \cite{ran2004fixed,gao2016hermitian}.

Well-established software packages, including widely used MATLAB routines for Lyapunov and Riccati equations, are integral to control system workflows, and their numerical performance has been extensively studied \cite{datta2004numerical,bini2011numerical,benner2020numerical}. In particular, MATLAB provides functions such as \texttt{lyap}, \texttt{dlyap}, \texttt{care}, \texttt{dare}, \texttt{idare}, \texttt{icare}, and \texttt{gdare} for solving various Lyapunov and algebraic Riccati equations.

	\begin{table}[h!]
	\begin{center}
		\begin{tabular}{|p{4.5cm} |p{1.3cm}|p{1.5cm}|}
			\hline
			Algorithm  & Convergence rate & Per-iteration complexity \\
			\hline
			Invariant/Deflating subspace &  -    &  \multirow{7}{*}{$\O(n^3) $}   \\ 
			\cline{1-2}
			Kleinman-Newton \cite{bini2011numerical}&  quadratic   &        \\
			\cline{1-2}
			Inexact Kleinman-Newton	\cite{feitzinger2009inexact} & quadratic     &         \\
			\cline{1-2}
			Structure-preserving doubling algorithms\cite{chu2005structure} & quadratic   &    \\
			\cline{1-2}
			Fixed-point \cite{dai2011eigenvalue}  & linear    &      \\
			\cline{1-2}
			Inverse free \cite{liu2022iterative}   & linear    &    \\
			\cline{1-2}
			
			Riemannian gradient descent	\cite{zhang2016first} & linear  &        \\
			\cline{1-2}
			RRSD  algorithm \cite{darmwal2023low} & linear  &        \\
			\hline	
			BMFC  algorithm \cite{han2024riemannian} & linear  &       $\O(n^2)$     \\
			\hline	
			R1RSD algorithm  [proposed] & linear   &       $\O(n^2 \log n)$     \\
			\hline	
		\end{tabular}
	\end{center}
	\caption{Comparison of solvers for \eqref{residual_norm}.}\label{sota_algo}
\end{table}

\section{Notation and Background}\label{paper.sec.2}
This section outlines the notation used in this work and presents the necessary background on Riemannian geometric concepts. 

\subsection{Notation}\label{notation} We denote vectors (matrices) by boldface lower (upper) case letters. The cardinality of  a set $\cI$ is denoted by $|\cI|$. The trace and transpose operations are denoted by $\tr{\cdot}$ and $(\cdot)^\T$, respectively.  
The lower triangular Cholesky factor of the symmetric positive definite matrix $\A$ is denoted by $\cL( \A)$, so that $\A = \cL(\A)\cL(\A)^\T$ and its largest eigenvalue is denoted by $\lambda_{1}(\A)$. The Euclidean gradient of a function $f:\Rn^{n\times n} \rightarrow \Rn$ is denoted by $\grd \;f(\X)$, while its Riemannian gradient is denoted by $\grd^R \;f(\X)$. The Euclidean and Frobenius norms are denoted by $\norm{\cdot}_2$ and  $\norm{\cdot}_F$, respectively. The tangent space at a point $\X \in \Pn$ is denoted by $T_\X\Pn$, with tangent vectors denoted by boldface Greek lower case letters, e.g., $\xib$. Given two tangent vectors $\xib$, $\etab \in T_\X\Pn$, their inner product is given by $\ipx{\xib}{\etab}$ and the corresponding norm is given by $\norm{\xib}_\X := \sqrt{\ipx{\xib}{\xib}}$.  
Given arbitrary $\X, \Y \in \Pn$, and the geodesic $\gamma(\lambda)$ joining them so that $\gamma(0)=\X$ and $\gamma(1)=\Y$, the tangent vector at $\X$ is denoted by $\xib_{\X\Y} := \gamma'(0)$. The power and exponential maps of an SPD matrix $\W$ with eigenvalue decomposition $\U\D\U^\T$ are given by 
\begin{align*}
	\W^k &= \U\D^k\U^\T & \exp(\W) &= \U\exp(\D)\U^\T = \sum_{k=0}^\infty \frac{\W^k}{k!}
\end{align*}
where $[\D^k]_{ii} = [\D]_{ii}^k$ and $[\exp(\D)]_{ii} = \exp([\D]_{ii})$ for all $1\leq i \leq n$. 
\subsection{Background on manifold optimization} 
Riemannian manifolds are nonlinear spaces equipped with a Riemannian metric (\cite{lee2018introduction}).  For the SPD manifold $\Pn = \Sn_{++}$, the tangent space  $T_\X\Pn$ can be identified with the set of symmetric matrices $\mathbb{S}^n$ (\cite{bridson2011metric}) via a natural isomorphism. We adopt the following Riemannian metric, commonly known as the affine-invariant metric:
\begin{eqnarray}
	\ipx{\xib}{\etab} = \tr{\X^{-1}\xib\X^{-1}\etab}. \label{metric} 
\end{eqnarray}
The choice of the metric in \eqref{metric} renders $\Pn$ a Hadamard manifold, i.e., a manifold with non-positive sectional curvature. In this context, it may be computationally efficient to use the Cholesky-inspired congruence mapping \cite{godaz2021vector}
\begin{align}
\cC_\X(\xib) = \B^{-1}\xib \B^{-\T} \label{mapping}	
\end{align}
where $\B = \cL(\X)$.  Under this mapping, the metric \eqref{metric} is reduced to the Euclidean inner product, i.e., $\ipx{\xib}{\etab} = \tr{\cC_\X(\xib)\cC_\X(\etab)}$. Further, the geodesic $\gamma:[0,1]\rightarrow \Pn$ starting at $\X \in \Pn$ in the direction of the tangent vector $\gamma'(0) = \xib$ is given by $	\gamma(\lambda) = \B\exp\left(\lambda\cC_\X(\xib)\right)\B^{\T}$~\cite{darmwal2023low}. Finally, $\Expx:T_\X\Pn \rightarrow \Pn$ is the exponential map such that $\Expx(\xib) = \gamma(1)$. Next, we introduce some  important definitions for the Riemannian manifold $\Pn$ with the Riemannian metric as specified in \eqref{metric}.
\begin{definition}[\textbf{Directional derivative}]
	Let $f:\Pn \rightarrow \Rn$ be a smooth function and $\gamma(\lambda):\mathbb{R}\rightarrow \Pn$ be a smooth curve satisfying $\gamma(0)=\X$ and $\gamma'(0)=\xib$.  The directional derivative of $f$  at $\X$ in the direction $\xib \in T_{\X}\Pn$ is the
	scalar \cite[p.~40]{absil2007optimization}:
	\begin{align}
		Df_{\X}(\xib)= \frac{d}{d\lambda}f(\gamma(\lambda))\Big|_{\lambda=0} 
	\end{align}
\end{definition}
\begin{definition}[\textbf{Riemannian gradient}]\label{Riemann_grad}
	The Riemannian gradient of a differentiable function $f:\Pn\rightarrow \Rn$ at $\X \in \Pn$  is defined as the unique tangent vector  $\grd^R \;f(\X) \in T_{\X}\Pn$ satisfying \cite[p.~46]{absil2007optimization}:
	\begin{align}
		Df_{\X}(\xib)&=\langle \grd^R \;f(\X),\xib\rangle_{\X}
	\end{align}
\end{definition}
The Riemannian and Euclidean gradients for $\Pn$ are related by \cite[p. 722]{Sra2015conic}:
\begin{align}
	\grd^R\; f(\X) = \X\grd \; f(\X)\X. \label{RiemanEuclideangrad}
\end{align} 

Lastly, the Cholesky version of the Riemannian gradient descent (RGD) update in the descent direction $\xib_{\X_t\X_{t+1}}=-\alpha_t\grd^R f\left(\X_t\right)$ is given by 
\begin{align}
\X_{t+1} &= \Expxt\left(\xib_{\X_t\X_{t+1}}\right) = \B_t\exp(\cC_{\X_t}(\xib_{\X_t\X_{t+1}}))\B_t^\T\nonumber\\
&= \B_t \exp(-\alpha_t\B_t^{-1}\grd^R f\left(\X_t\right)\B_t^{-\T})\B_t^{\T}.\label{gdupdate}	
\end{align}
Observe that since $\xib_{\X_t\X_{t+1}} \in \Sn$, it follows that $\cC_{\X_t}(\xib_{\X_t\X_{t+1}}) \in \Sn$ which ensures that $\X_{t+1} \in \Pn$. 


\section{Riemannian subspace descent algorithm}\label{rsd_algo}
In this section, we describe the proposed Riemannian subspace descent algorithm and show how it supports greedy direction selection while keeping the per-iteration cost to $\O(n^2)$ flops. Different from the classical RGD in \eqref{gdupdate}, subspace descent algorithms update the iterate along a low-dimensional subspace, e.g., along one or a few coordinates, rather than moving along the full Riemannian gradient, thereby reducing the computational burden at each iteration. In the present case, we express the (transformed) Riemannian gradient as a sum of rank-one eigen-components and then update the iterate using only one selected component at each step. This decomposition allows us to identify the most promising descent component at low computational cost and efficiently maintain $\X_t$, $\X_t^{-1}$, and the Cholesky factor $\B_t = \cL(\X_t)$ via rank-one updates. We remark that the proposed approach of selecting a direction from a structured set of candidate directions differs from the subspace/coordinate descent schemes in  \cite{celledoni2008descent, shalit2014coordinate, gao2018new, gutman2023coordinate} which typically fix a basis for $T_\X\Pn$ and then choose the update direction by projecting the Riemannian gradient onto the resulting subspace. The method is, however, closely related to  \cite{darmwal2023low}, and we discuss this connection at the end of the section. 

\subsection{Subspace selection}
The subspace selection is motivated by the Spectral theorem: any symmetric matrix $\G$ admits an orthogonal eigen-decomposition $\G = \sum_{i=1}^n \lambda^{(i)} \u^{(i)}(\u^{(i)})^\T$. In other words, $\G$ can be expressed as a sum of $n$ rank-one terms $\u^{(i)}(\u^{(i)})^\T$ that are orthogonal in the Euclidean inner product, i.e., $\tr{\u^{(i)}(\u^{(i)})^\T\u^{(j)}(\u^{(j)})^\T} = 0$ for $i\neq j$ and 1 for $i=j$. These properties suggest that eigen-decomposition can be used to obtain subspaces along which an iterate may be updated. Specifically, we consider the eigen-decomposition of the transformed Riemannian gradient $\cC_{\X_t}(\grd^R \; f(\X_t))$, which appears within the exponential in \eqref{gdupdate}: 
\begin{align*}
	\cC_{\X_t}(\grd^R f(\X_t)) &\eqtext{\eqref{RiemanEuclideangrad}} \B_t^\T \grd f(\X_t) \B_t = \sum_{i=1}^n \lambda_t^{(i)}\x_t^{(i)}(\x_t^{(i)})^\T
\end{align*}
where $(\lambda_t^{(i)}, \x_t^{(i)})$ is an eigenpair. This eigen-decomposition induces the following Riemannian orthogonal decomposition of $\grd^R f(\X_t)$:
\begin{align}
	\grd^R f(\X_t) = \sum_{i=1}^n \lambda_t^{(i)}\H_t^{(i)}
\end{align}
where the rank-one tangent directions $\H_t^{(i)} := \B_t\x_t^{(i)}(\x_t^{(i)})^\T\B_t^\T$ are orthonormal with respect to the affine-invariant metric \eqref{metric} since
\begin{align}
	\langle \H_t^{(i)}, &\H_t^{(j)}\rangle_{\X_t} = \tr{\cC_{\X_t}(\H_t^{(i)})\cC_{\X_t}(\H_t^{(j)})} \nonumber\\
	&=\tr{\x_t^{(i)}(\x_t^{(i)})^\T\x_t^{(j)}(\x_t^{(j)})^\T} = ((\x_t^{(i)})^\T\x_t^{(j)})^2
\end{align}
which is one for $i = j$ and zero otherwise. Consequently, if $\grd^Rf(\X_t) \neq \zero$, then $-\lambda_t^{(i)}\H_t^{(i)}$ for any $\lambda_t^{(i)}\neq 0$ is a valid descent direction since
\begin{align*}
	Df_{\X_t}(-\lambda_t^{(i)}\H_t^{(i)}) &= -\langle  \sum_{j=1}^n \lambda_t^{(j)}\H_t^{(j)}, \lambda_t^{(i)}\H_t^{(i)}\rangle_{\X_t} \\
	&=-(\lambda_t^{(i)})^2 < 0.
\end{align*}
Indeed, since the magnitude of the directional derivative is exactly $(\lambda_t^{(i)})^2$, choosing the candidate direction corresponding to the dominant eigenvalue, i.e., one with the largest magnitude, results in the steepest first-order decrease. If $(\lambda_t^{(1)},\x_t^{(1)})$ denotes the dominant eigenvalue-eigenvector pair, then the proposed greedy subspace descent update can be written as
\begin{align}
	\X_{t+1} = \B_t\exp\left(-\beta_t\lambda_t^{(1)} \x_t^{(1)}(\x_t^{(1)})^\T\right)\B_t^\T \label{gsupdate}
\end{align}
where $\beta_t > 0$ is the step-size. This corresponds to choosing $\xib_{\X_t\X_{t+1}} = -\beta_t\lambda_t^{(1)}\H_t^{(1)}$. We next discuss the efficient implementation of \eqref{gsupdate} and comment on its computational complexity. 

\subsection{Efficient Updates}
In the current form, naive implementation of the update in \eqref{gsupdate} is still challenging as it seems to involve several operations whose complexity is generally $\O(n^3)$: matrix exponential, Cholesky factorization, eigenvalue decomposition, and dense matrix multiplications. Even with $\X_t$ and its Cholesky factor $\B_t$ available, forming the transformed Riemannian gradient $\B_t^\T\grd f(\X_t) \B_t$ is not generally possible in $\O(n^2)$ time. Here, we show that \eqref{gsupdate} can be written as a rank-one  update, which allows all required intermediates to be maintained via low-rank updates and hence implemented with $\O(n^2)$ per-iteration cost. 


The key observation in \eqref{gsupdate} is that the matrix exponential of a rank-one matrix is easy to calculate, and yields the modified rank-one update rule \cite[Example 1.2.5]{higham2008functions}:
\begin{align}
	\X_{t+1} &= \X_t + \left[\exp(-\lambda_t^{(1)}\beta_t)-1\right]\B_t\x_t^{(1)}(\x_t^{(1)})^\T\B_t^\T \label{xupdate}
\end{align}
which, given $\B_t$ and $\x_t$, requires only $O(n^2)$ flops. Next, the rank-one modification in \eqref{xupdate} allows the Cholesky factor to be updated efficiently. In particular, $\B_{t+1}$ can be obtained from $\B_t$ using only $O(n^2)$ flops via a standard rank-one Cholesky update procedure \cite[Sec. 3]{gill1974methods}, also available as MATLAB function \texttt{cholupdate}. 

Calculation of the dominant eigenvalue-eigenvector pair is accomplished using the power method \cite{demmel1997applied}. Starting with an random unit norm $\y_t^{(0)} \in \Rn^n$, the power iterations take the form: 
\begin{align}
\y_t^{(i)}=\frac{\B_t^{\T}\grd f(\X_t)\B_t\y_t^{(i-1)}}{\|\B_t^{\T}\grd f(\X_t) \B_t\y_t^{(i-1)}\|_2}
\end{align}
for $i =1, 2, \ldots$. We terminate the method after $r$ iterations and use $\y_t = \y_t^{(r)}$ as an approximation of $\x_t$ in \eqref{xupdate}. The corresponding approximate eigenvalue can be calculated using the Rayleigh quotient as $\lambda_t = \y_t^\T\B_t^\T\grd f(\X_t)\B_t\y_t$. Hence, the approximate descent direction becomes
\begin{align}\label{xib}
	\xib_{\X_t\X_{t+1}}= -\beta_t\lambda_t \B_t\y_t\y_t^\T\B_t^\T
\end{align}
Here, observe that if we can apply $\grd f(\X_t)$ to a vector in $\O(n^2)$ flops, each power iteration can be implemented using only three $\O(n^2)$ operations: namely $\v_t^{(i)} = \B_t\y_t^{(i)}$, $\w_t^{(i)} = \grd f(\X_t) \v_t^{(i)}$, and $\tilde{\y}_t^{(i)} = \B_t^\T\w_t^{(i)}$, followed by normalization. Hence, the overall per-iteration complexity of the proposed algorithm is $\O(rn^2)$. The pseudo-code of the complete algorithm is summarized in Alg. \ref{sd_algo}.

\begin{algorithm}
	\caption{Rank-one Riemannian Subspace Descent (R1RSD) algorithm}\label{sd_algo}
	\begin{algorithmic}[1]
		\renewcommand{\algorithmicrequire}{\textbf{Input:}}
		\REQUIRE $\X_0 = \B_0\B_0^\T$, $T$, $\{\beta_t\}$
		\FOR {$t = 0$ to $T-1$}
		\STATE Calculate $\grd\; f(\X_t)$
		\STATE Draw $\u_t \sim \cN(\zero,\I_n)$ and set $\y_t^{(0)} = \u_t/\norm{\u_t}$
		\FOR {$i = 1$ to $r$}
		\STATE $\y_t^{(i)}=\frac{\B_t^{\T}\grd f(\X_t)\B_t\y_t^{(i-1)}}{\|\B_t^{\T}\grd f(\X_t) \B_t\y_t^{(i-1)}\|_2}$
		\ENDFOR
		\STATE Set $\y_t=\y_t^{(r)}$
		\STATE $\lambda_t=\y_t^{\T}\B_t^{\T}\grd\; f(\X_t)\B_t\y_t$
		\STATE $\X_{t+1} =\X_t +\left[\exp(-\lambda_t\beta_t)-1\right]\B_t\y_t\y_t^{\T}\B_t^{\T}$
		\STATE Update $\B_t$ (e.g. using \texttt{cholupdate} in MATLAB)
		\ENDFOR
		\RETURN $\X_T$ 
	\end{algorithmic} 
\end{algorithm} 

\subsection{Calculation of $\grd\; f(\X_t)\v_t^{(i)}$}
In this section, we specify the class of functions $\cF$ for which the gradient-vector product $\grd\; f(\X_t)\v_t^{(i)}$ can be computed in $\O(n^2)$ time for any SPD iterate $\X_t$ and vector $\v_t^{(i)}$. We will maintain intermediate variables and matrix-vector products so as to avoid matrix-matrix multiplications required to explicitly form $\grd f(\X_t)$. The key is to restrict $f$ to compositions of basic matrix operations that (a) preserve symmetry and positive definiteness where required, and (b) admit efficient low-rank updates and matrix-vector application. 

To this end, we represent $f$ using a composition of several elementary layers. Specifically, we start with some tuple containing $\X$ and constant matrices $\A$, $\B$, $\C$, etc.: 
\begin{align}
	\M^{(0)}(\X) := (\X,\A,\B, \C, \ldots)
\end{align}
and apply $m$ compositional operations of the form 
\begin{align}
\M^{(\ell)}(\X) &= \cT^{(\ell)}(\M^{(\ell-1)}(\X)), &	\ell = 1, \ldots, m.
\end{align}
Here, the operation $\cT^{(\ell)}$ takes $\M^{(\ell-1)}(\X)$ as input and produces an output tuple $\M^{(\ell)}(\X)$, whose elements can be obtained by copying, transposing, adding, multiplying, or inverting elements of the input tuple. Specifically, the $j$-th element of $\M^{(\ell)}(\X)$ is given by one of the following operations ($i$ and $k$ are indices of the elements of $\M^{(\ell-1)}(\X)$): 
\begin{itemize}[leftmargin=*]
	\item Scaling: $\M_j^{(\ell)} = c\M_i^{(\ell-1)}$;
	\item Scaling and Transpose:  $\M_j^{(\ell)} = c(\M_i^{(\ell-1)})^{\T}$;
	\item Sum: $\M_j^{(\ell)} = \sum_i c_i\M_i^{(\ell-1)}(\X)$;
	\item Product: $\M_j^{(\ell)} = c_j\prod_q \M_{i_q}^{(\ell-1)}(\X)$ where $i_q$ are indices of the elements of $\M^{(\ell-1)}(\X)$; and
	\item Inverse: $\M_j^{(\ell)} = (\M_i^{(\ell-1)})^{-1}$ where $\M_i^{(\ell-1)} \succ \zero$; 
\end{itemize}
Finally, we apply scalar functions to the $J$ components of $\M^{(m)}(\X)$, so as to form 
\begin{align}
	f(\X) = \sum_{j=1}^J f_j(\M_j^{(m)}(\X)) \label{fjx}
\end{align}
where $f_j$ is one of the following functions (a) trace, (b) squared Frobenius norm, or (c) log-det if $\M_j^{(m)}(\X) \succ \zero$. As a simple example, consider the NME $f(\X) = \norm{\R(\X)}_F^2$ where $\R(\X) = \X+\A^\T\X^{-1}\A-\Q$, for which the different layers are given by:
\begin{align}
	\M^{(0)}(\X) &= (\X, \A, \Q) \\
	\M^{(1)}(\X) &= (\X, \A^\T, \X^{-1}, \A, -\Q) \\
	\M^{(2)}(\X) &= (\X,\A^\T\X^{-1}\A ,- \Q) \\
	\M^{(3)}(\X) &= (\X + \A^\T\X^{-1}\A - \Q) \\
	f(\M^{(3)}(\X)) & = \|\M^{(3)}(\X)\|_F^2.
\end{align}

Given the compositional form, we can find the gradient of $f$ with respect to $\X$ using chain rule to each layer, and write the full gradient as sum of $\tilde{J}$ terms: 
\begin{align}
	\grd f(\X) &= \sum_{j=1}^{\tilde{J}} \tilde{\M}_j^{(\tilde{m})}(\X) \\
	\tilde{\M}^{(\tilde{m})}(\X) &= \tilde{\cT}^{(\tilde{m})}(\tilde{\cT}^{(\tilde{m}-1)}(\ldots\tilde{\cT}^{(0)}(\M^{(0)}(\X))))
\end{align}
so that the gradient graph has layers $\tilde{\M}^1, \ldots, \tilde{\M}^{(\tilde{m})}$, and the components of the last layer are summands of $\grd f(\X)$. We note that each layer in the gradient function is also formed using the same primitive operations, namely sum, product, transpose, and inverse. Hence, we obtain the product as
\begin{align}
	\grd f(\X)\v = \sum_{j=1}^{\tilde{J}} \tilde{\M}_j^{(\tilde{m})}(\X)\v
\end{align}
where each term is evaluated using only matrix-vector products and application of inverses, if needed, that must be separately maintained using low-rank updates. For the NME example earlier, we have that
\begin{align}
	\grd f(\X) = 2\R(\X) - 2\X^{-1}\A\R(\X)\A^\T\X^{-1}
\end{align}
for which the different layers are given by:
\begin{align}
	\tilde{\M}^{(1)}(\X) &= (\X, \A^\T, \X^{-1}, \A, -\Q) \\
	\tilde{\M}^{(2)}(\X) &= (\X,\A^\T\X^{-1}\A ,- \Q, \X^{-1}\A, \A^\T\X^{-1}) \\
	\tilde{\M}^{(3)}(\X) &= (\R(\X),\X^{-1}\A, \A^\T\X^{-1}) \\
	\tilde{\M}^{(4)}(\X) &= (2\R(\X),-2\X^{-1}\A\R(\X)\A^\T\X^{-1}) 
\end{align}
As there is only one inverse operator, we only need to maintain:
\begin{align}
	\X_{t+1}^{-1}
	&= \X_t^{-1}
	+ \left[\exp\!\left(\lambda_t^{(1)} \beta_t\right) - 1\right]
	\B_t^{-\T} \x_t^{(1)} \big(\x_t^{(1)}\big)^\T \B_t^{-1} \label{x_inv_upate}
\end{align}
for each iteration, which incurs $\O(n^2)$ cost. Once $\X_{t+1}^{-1}$ is available, the gradient-vector product can be calculated using matrix-vector products alone, and without needing any other $\O(n^3)$ operations, as shown in Appendix \ref{nmestep}.

More generally, inverse nodes are the sole source of potential $\O(n^3)$ cost. By maintaining the outputs of all inverse nodes, via low-rank updates, every term $\tilde{\M}_j^{(\tilde m)}(\X)\v$ can be evaluated using only matrix–vector products, yielding an overall $\O(n^2)$ gradient–vector product.

To summarize, for any $f\in\cF$ specified by the primitive-layer construction above, the gradient admits a compositional structure built from the same primitives. The product $\nabla f(\X_t)\v_t^{(i)}$ is obtained by evaluating the gradient graph in a matrix--vector fashion, i.e., by applying each terminal summand $\tilde{\M}_j^{(\tilde m)}(\X_t)$ to $\v_t^{(i)}$ and summing the results. Outputs of inverse nodes are maintained across iterations via low-rank updates and the full gradient-vector product is calculated in $\O(n^2)$ time.


\section{Performance Analysis}
In this section, we analyze the performance of the proposed algorithm under certain regularity conditions on the objective function $f$. To this end, we will characterize the iteration complexity of Alg. \ref{sd_algo}, which is the number of iterations required to ensure that $\X_t$ is $\epsilon$-stationary on average. Compared to the analysis of RGD and subspace descent variants in the literature, the key complication here is that the random initialization required by the power method makes the subsequent steps as well as the overall behavior of the algorithm random. Hence, the performance analysis must involve expectations. We use $\Et{\cdot}$ to denote the expectation with respect to $\u_t$ (cf. Step 3 in Alg. \ref{sd_algo}) and $\EE[\cdot]$ to denote full expectation. We begin with stating the key assumptions, which are standard in the RGD and power method literature. 

	
\begin{assumption}\label{a-strong_convex}
	The function $f:\Pn \rightarrow \Rn$ is geodesically $\mu$-strongly convex, i.e., for any two arbitrary points $\X$, $\Y \in \Pn$, and  connecting  geodesic starting at $\X$ with $\gamma'(0) = \xib_{\X\Y}$, it holds that (\cite{zhang16PMLR}):
	\begin{align}
		f(\Y)&\geq  f(\X)+\langle \grd^R f(\X), \xib_{\X\Y}\rangle_{\X} + \frac{\mu}{2}\|\xib_{\X\Y}\|_{\X}^2.  \label{geodesic_strong_convx}
	\end{align}
\end{assumption}	
		
\begin{assumption}\label{a-smooth}
	The function $f : \Pn\rightarrow \mathbb{R}$  is geodesically $L$-smooth, i.e., for any two arbitrary points $\X$, $\Y \in \Pn$, and  connecting  geodesic starting at $\X$ with $\gamma'(0) = \xib_{\X\Y}$, it holds that (\cite{zhang16PMLR}):
	\begin{align}
		f(\Y) \leq f(\X) +\langle \grd^R f(\X), \xib_{\X\Y}\rangle_{\X} + \frac{L}{2}\|\xib_{\X\Y}\|_{\X}^2. \label{Lipschitz_cont_grad}
	\end{align}
\end{assumption}

Of these, the geodesic smoothness assumption is quite standard and applies to all the examples discussed so far. On the other hand, the geodesic strong convexity assumption is stronger. It clearly holds for the residual norm in the case of the Lyapunov function, but is difficult to verify for the other equations in Sec. \ref{intro}. The bounds will also depend on the initialization through $D_0:=f(\X_0)$. For the sake of convenience, the transformed gradient is denoted as
\begin{align}
	\P(\X) := \cL(\X)^\T\grd f(\X)\cL(\X)
\end{align}
for any $\X \in \Pn$, so that $\P(\X_t) = \B_t^\T\grd f(\X_t)\B_t$. The final assumption is standard for analyzing the power method. 

\begin{assumption}\label{gap}
	The spectral gap of $\P(\X)$ is non-trivial for all $\X$, i.e.,  $\frac{\abs{\lambda^{(2)}(\X)}}{\abs{\lambda^{(1)}(\X)}} \leq \rho < 1$ for some $\rho$ and all $\X\in \Pn$, where the eigenvalues of $\P(\X)$ arranged such that $\lambda^{(1)}(\X) > \ldots \geq \lambda^{(n)}(\X)$.
\end{assumption}

We remark the gap-free analysis of power method is also possible but requires stronger assumption of positive semi-definiteness of $\P(\X)$ \cite[Sec. 6]{martinsson2020randomized}, which is unlikely to hold for all $\X$ in the current setting. Hence, we resort to gap-dependent analysis, though Assumption \ref{gap} may be difficult to verify in practice. Before establishing the required results, we state the following preliminary lemma that lower bounds the approximate Rayleight quotient $\lambda_t^2$ obtained in Step 8 of Alg. \ref{sd_algo}. 

\begin{lemma}\label{lamt}
	Under Assumption \ref{gap} for $\P_t = \P(\X_t)$, after $r \geq \frac{1}{2}\log_{\frac{1}{\rho}}(8n)$ iterations of the power method, it holds that
	\begin{align}\label{lamteq}
	\Et{\lambda_t^2} &\geq \frac{1}{48}(\lambda_t^{(1)})^2.
	\end{align}
\end{lemma}
\begin{IEEEproof}
	Writing $\y_t^{(0)} = \sum_i c_i \x_t^{(i)}$ for some coordinates $c_i$, we have that \cite[Thm. 8.3.1]{golub2013matrix}:
	\begin{align}
		\abs{\lambda_t-\lambda_t^{(1)}}\leq \max_{2\leq i \leq n} \abs{\lambda_t^{(1)} - \lambda_t^{(i)}} \rho^{2r}\tan^2(\theta_0)
	\end{align}
	where $\tan^2(\theta_0) = \frac{1-c_1^2}{c_1^2}$. Since $\abs{\lambda_t^{(1)} - \lambda_t^{(i)}} \leq 2\abs{\lambda_t^{(1)}}$, we have
	\begin{align}
		\abs{\lambda_t} \geq \abs{\lambda_t^{(1)}} - \abs{\lambda_t - \lambda_t^{(1)}} \geq \abs{\lambda_t^{(1)}}\left[1-2\tan^2(\theta_0)\rho^{2r}\right]
	\end{align}
	which is vacuous if the right-hand side is non-positive. To convert this into a bound in expectation, let us consider the random variable $\sZ$ which takes the value one when $c_1^2 \geq \frac{1}{2n}$ and zero otherwise. If $\sZ= 1$, we have that $\tan^2(\theta_0) = \frac{1}{c_1^2}-1 \leq 2n$, which yields the bound
	\begin{align}\label{lamtbound1}
		\Et{\lambda_t^2} \geq \Et{\sZ}(\lambda_t^{(1)})^2(\max\{0,1-4n\rho^{2r}\})^2.
	\end{align}
	Finally to bound $\Et{\sZ}$, we note that $c_1^2 \sim \text{Beta}(\frac{1}{2},\frac{n-1}{2})$ \cite{frankl1990some} so that $\Et{c_1^2} = \frac{1}{n}$ and $\Et{c_1^4} = \frac{3}{n(n+2)}$. Application of the Paley–Zygmund inequality \cite{petrov2007lower} yields $\mathbb{P}[c_1^2 \geq \frac{1}{2n}] \geq \frac{1}{4}\frac{n+2}{3n} \geq \frac{1}{12}$, which upon substituting in \eqref{lamtbound1} gives
	\begin{align}
		\Et{\lambda_t^2} \geq	\frac{1}{12}(\lambda_t^{(1)})^2(\max\{0,1-4n\rho^{2r}\})^2.
	\end{align} 
	Finally, for $r = \frac{1}{2}\log_{\frac{1}{\rho}}(8n)$, we get $4n\rho^{2r} = \frac{1}{2}$, which results in the desired bound. 
\end{IEEEproof}

\subsection{Minimizing geodesically smooth functions}
We first consider the general case of geodesically $L$-smooth but possibly non-convex functions. In this case, $f$ may have multiple stationary points that may not be a global minimum. However, as with all first-order algorithms, the proposed subspace descent algorithm can only guarantee an $\epsilon$-stationary point, i.e., one that satisfies: 
\begin{align}\label{epsstat}
	\EE\left[\min_{0 \leq t \leq T-1} \norm{\grd^Rf(\X_t)}_{\X_t}\right] \leq \epsilon.
\end{align} 
We know that unless $f(\X_t)$ is close to zero, the iterate $\X_t$ is not globally optimum and does not solve \eqref{eqn_form}. In such case, we may try again by restarting the algorithm from a different random initialization. The following theorem characterizes the iteration complexity of the proposed subspace descent algorithm. 

\begin{theorem}\label{thm1}
 	Under Assumptions \ref{a-smooth} and \ref{gap}, the proposed algorithm yields an $\epsilon$-stationary point in $\O(\frac{nL}{\epsilon^2})$ iterations and each iteration incurs $\O(n^2\log(n))$ flops. 
 \end{theorem}

\begin{proof}
	For the sake of simplicity, we will use the step size $\beta_t = 1/L$. We begin with using \eqref{Lipschitz_cont_grad} between $\X_t$ and $\X_{t+1}$ as well as the definition of $\xib_{\X_t\X_{t+1}}$ in \eqref{xib} to obtain
	\begin{align}
		&f(\X_{t+1}) - f(\X_t) \\
		&\leq \ip{\grd^Rf(\X_t)}{\xib_{\X_t\X_{t+1}}}_{\X_t} + \tfrac{L}{2}\|\xib_{\X_t\X_{t+1}}\|_{\X_t}^2 \nonumber\\
		&\hspace{-2mm}\leqtext{\eqref{metric},\eqref{RiemanEuclideangrad},\eqref{xib}} \hspace{-3mm}-\tfrac{\lambda_t}{L}\tr{\X_t\grd\; f(\X_t)\X_t\X_t^{-1}\B_t\y_t\y_t^\T\B_t^\T\X_t^{-1}} \nonumber\\
		&+ \tfrac{\lambda_t^2}{2L}\tr{\B_t\y_t\y_t^\T\B_t^{\T}\X^{-1}\B_t\y_t\y_t^\T\B_t^{\T}\X_t^{-1}}\\
		&= -\tfrac{\lambda_t}{L}\tr{\y_t^\T\B_t^\T\grd\; f(\X_t)\B_t\y_t}  + \tfrac{\lambda_t^2}{2L} = -\tfrac{\lambda_t^2}{2L}.\label{fdec}
	\end{align}
	That is, the per-iteration decrease in the function value is proportional to the square of the approximate Rayleigh quotient obtained from the power method.  
		
	Next, we bound the Riemannian gradient norm using \eqref{RiemanEuclideangrad}:
	\begin{align}
		&\|\grd^R f(\X_t)\|_{\X_t}^2 \eqtext{\eqref{metric},\eqref{RiemanEuclideangrad}} \tr{\grd f(\X_t) \X_t \grd f(\X_t)\X_t} \nonumber\\
		 &=\tr{(\B_t^\T\grd \; f(\X_t)\B_t)^2} = \sum\nolimits_i (\lambda_t^{(i)})^2 \leq n (\lambda_t^{(1)})^2 \label{gradbound}
	\end{align}
	Taking expectation and applying the bound in Lemma \ref{lamt} for $r \geq \frac{1}{2}\log_{\frac{1}{\rho}}(8n)$, we obtain
	\begin{align}
		\EE\|\grd^R f(\X_t)\|_{\X_t}^2 &\leq n \Ex{(\lambda_t^{(1)})^2} \leqtext{\eqref{lamteq}} 48 \Ex{\lambda_t^2} \\
		&\hspace{-1cm}\leqtext{\eqref{fdec}} 96nL(\Ex{f(\X_t)} - \Ex{f(\X_{t+1})}). \label{gradbound2}
	\end{align}
	Therefore, taking sum over $t = 1, \ldots, T$, dividing by $T$, and using the fact that $f(\X_T) \geq 0$, we obtain
	\begin{align}
		\Ex{\min_t \|\grd^R f(\X_t)\|_{\X_t}^2} &\leq \frac{1}{T}\sum_{t=0}^{T-1} \EE\|\grd^R f(\X_t)\|_{\X_t}^2 \\
		&\leq \tfrac{96nLD_0}{T}
	\end{align}
	and subsequently from Cauchy-Schwarz inequality, we get
	\begin{align}
		\Ex{\min_t \|\grd^R f(\X_t)\|_{\X_t}} \leq \sqrt{\frac{96nLD_0}{T}}.
	\end{align}
	Equating the expression on the right with $\epsilon$ yields the required iteration complexity. We note that for $r = \Theta(\log(n))$, the total flop count of the proposed algorithm becomes $\O(\frac{n^3\log(n)L}{\epsilon})$ since each iteration incurs $\Theta(n^2\log(n))$ flops. 
\end{proof}

The result in Thm. \ref{thm1} improves upon the $\O(\frac{\abs{\cI}L}{\epsilon^2})$ iteration complexity of RCD algorithm in \cite[Thm. 4.1]{han2024riemannian} since the cardinality of the index set is $\abs{\cI} = \O(n^2)$. Likewise, the total flop count of Alg. \ref{sd_algo} is $\O(n^3\log(n))$ which is worse than the RGD flop count of $\O(n^3)$ but better than the $\O(n^4)$ count of RCD algorithm. 

\subsection{Minimizing geodesically strongly convex and smooth functions}
When $f$ is geodesically strongly convex, \eqref{residual_norm} has a unique stationary point that solves \eqref{eqn_form}, i.e., $f(\X^\star) = 0$. In this case, it suffices to characterize the number of iterations required to ensure that $f(\X_T)\leq \epsilon$. The following theorem provides required iteration complexity bound in terms of $\kappa = L/\mu$. 
\begin{theorem}\label{thm2}
	Under Assumption \ref{a-strong_convex}-\ref{a-smooth}, the proposed algorithm has an iteration complexity of $\O\left(n\kappa\log\left(\frac{D_0}{\epsilon}\right)\right)$.
\end{theorem}
\begin{proof}
	Applying \eqref{geodesic_strong_convx} to $f$ from $\X \in \Pn$ to $\X^\star$ and using the fact that $f(\X) \geq f(\X^\star) = 0$, we obtain
	\begin{align}
		\frac{\mu}{2}\|\xib_{\X\X^{\star}}\|_{\X}^2 &\leq f(\X^{\star}) - f(\X) - \ip{\grd^R f(\X)}{\xib_{\X\X^{\star}}}_{\X}\nonumber\\
		&\leq \norm{\grd^R f(\X)}_\X\norm{\xib_{\X\X^{\star}}}_{\X}\label{grad_bound_intermediate}
	\end{align}
	where we have used the Cauchy-Schwarz inequality. Hence, if $\X \neq \X^\star$, it follows that $\|\grd^R f(\X)\|_{\X} \geq  \frac{\mu}{2}\|\xib_{\X\X^\star}\|_{\X}$. Substituting back into \eqref{grad_bound_intermediate} and using $f(\X^\star) = 0$, we obtain
	\begin{align}
		\frac{2}{\mu}\|\grd^R f(\X)\|_{\X}^2 \geq  f(\X)\label{grad_dominate}
	\end{align}
	Writing \eqref{grad_dominate} for $\X = \X_t$, taking expectations, and combining with \eqref{gradbound2}, we obtain
\begin{align}
	\Ex{f(\X_{t+1})}  &\leq \left(1-\tfrac{\mu}{192nL}\right)\Ex{f(\X_t)} \\
	&\leq \left(1-\tfrac{\mu}{192nL}\right)^{t+1}D_0 \label{linconv}
\end{align}
Hence, to ensure that $\Ex{f(\X_T)} \leq \epsilon$, we need $T = \O(n \kappa \log(D_0/\epsilon))$. 
\end{proof}
Observe that Thm. \ref{thm2} has the same dependence on $n$ as Thm. \ref{thm1}, but improved dependence on $\epsilon$. The overall flop count of $\O(n^3\kappa \log(nD_0/\epsilon))$ is slightly worse than that of classical RGD, but the proposed algorithm has a lower per-iteration complexity allowing it to be applied to larger scale problems. We also note that the RRSD variants proposed in \cite{darmwal2023low} also incur a per-iteration complexity of $\O(n^3)$ when applied to \eqref{residual_norm}. In terms of total flop counts, the faster RRSD-multi algorithm in \cite{darmwal2023low} requires $\O(n^4 \log(D_0/\epsilon))$ flops in the geodesically strongly convex and smooth case. 

\section{Numerical experiments}
In this section, we compare the numerical performance of the proposed algorithm against state-of-the-art algorithms in solving CARE, DARE,  and an NME from \cite{huang2018structure}:
\begin{align}
	&\A^{\T}\X+\X\A-\X\G\X+\H=0  \quad  &(\text{CARE}) \label{care}\\
	&\X-\A^\T\X(\I+\G\X)^{-1}\A-\Q=0 \quad  &(\text{DARE})\label{dare}\\
	&\X+\A^{\T}\X^{-1}\A=\Q   \quad  &(\text{NME}).\label{nme}
\end{align}
Of these CARE frequently appears in optimal control problems~\cite{liberzon2012calculus, troltzsch2010optimal}. Consider, for instance, a standard infinite-dimensional optimal control problem from \cite{benner2016inexact}, but instead using an energy-based cost function 
\begin{align}\label{carecost}
	J &= \min \int_{0}^{\infty} \left( \gamma \int_{\Omega_0} \tilde{\x}^2(\xi,t)\,d\xi \right) + u^2(t)\,dt.
\end{align}
Such cost functionals are common in physical partial differential equation (PDE) models, such as in optimal control for thermal management in semiconductor devices~\cite{wang2013modeling, zanini2009control, salvi2021review}, control of pollutant emissions from chimneys~\cite{manzoni2021optimal}, and incompressible flow problems~\cite{bansch2012stabilization}, among other applications.

The performance of all these algorithms is compared with the state-of-the-art Burer-Monteiro factorization (BMFC) algorithm from \cite{han2024riemannian} for subspace descent. We also compare with the standard structure-preserving doubling algorithm (SDA)~\cite{huang2018structureI} and the MATLAB built-in \texttt{icare} function. Unlike the subspace descent methods, both SDA and \texttt{icare} are very fast for small-scale problems ($n \ll 5000$) but fail to run for large $n$. Similarly, for DARE and NME, full-update state-of-the-art algorithms, such as fixed-point, SDA, and MATLAB function \texttt{idare}, do not scale to large $n$ and are therefore not included in the results. We remark that if $\X$ is sparse or low-rank, we can solve CARE more efficiently, such as using algorithms from \cite{li2002low, benner2009adi, benner2015matrix, benner2016inexact}. However, these algorithms are not efficient for the general case, such as that arising when solving CARE for the energy-based cost function considered in \eqref{carecost}.

The residual error of the matrix equations is used as the performance metric, while the number of iterations and the per-iteration wall-clock time serve as measures of algorithmic efficiency. All the runtime results, including per-iteration and total execution times, are reported for the same machine with 8 GB RAM. However, other experiments were conducted on systems with higher RAM as well. 

\subsection{Implementation details}
For the proposed as well as BMFC algorithms, we implement all the steps carefully so as to ensure that the per-iteration cost remains $\O(n^2)$. To this end, we utilized the Woodbury identity to perform rank-one updates and maintained appropriate inverse matrices where required. The details of the updates are provided in the Appendices \ref{stepcare}, \ref{nmestep}, and \ref{stepdare}. For the purpose of implementation, we write the update as $\X_{t+1} = \X_t +\alpha_t \B_t\y_t\y_t^\T\B_t^\T$ and directly tune  $\alpha_t$ via line search. Interestingly, for the proposed algorithm, in all the three problems, the low-rank structure of the updates allows us to select the step-sizes via exact line search. Specifically, it is possible to find the step size that minimizes the objective along the descent direction without significant additional effort.

Exact line search is also possible for the BFMC algorithm for solving CARE and NME, but not for DARE, where Armijo line search must instead be used. 
The Armijo line search method iteratively decreases the stepsize $\delta$ by a factor $p$, until the following sufficient decease condition is satisfied along the descent direction $\D$:
\begin{align}
	f(\Y_{t+1})\leq f(\Y_{t})+\beta\delta\;
	\tr{\grd f(\Y)^\T \D}
\end{align}
We tested the performance of BMFC for different values of $\beta$ and $p$. For the cyclic version of BMFC, varying $\beta \in \{1, 0.1, 0.01, 0.001\}$ showed no noticeable effect on convergence and hence we set $\beta = 0.01$. We then tested for $p \in \{1/2, 1/4, 1/8, 1/16, 1/32, 1/64, 1/128\}$ and found that $p = 1/8$ provides the best performance in terms of both total iterations and the number of Armijo line-search steps. For sampling
with replacement, the choice $\beta = 0.1$ and $p = 1/8$ yielded the best overall performance.


\subsection{Per-iteration performance and design choices}
In this subsection, we empirically evaluate the per-iteration computational complexity of the proposed method and compare it with state-of-the-art algorithms, including BMFC and SDA. We also investigate the impact of direction-selection and step-size strategies on overall performance. Where applicable, the other parameters are manually tuned individually for all algorithms. 

\subsubsection{Per-iteration times} Figure~\ref{per_iter_time_RSD_BM_SDA} compares the single-iteration wall-clock time of SDA, BMFC, and R1RSD for solving DARE, measured on a system with 256~GB RAM to reduce simulation time. The BMFC algorithm employs an Armijo line-search method with tuned parameters. For each problem instance, we execute 200 iterations and report the mean per-iteration time. The results show that BMFC incurs a significantly higher per-iteration cost than the proposed R1RSD algorithm with 10 power iterations. This overhead arises primarily from the line-search procedure, as BMFC requires, on average, five Armijo line-search steps per iteration.
\begin{figure}[H]
\includegraphics[width=\columnwidth, trim = .5cm 2cm 1cm 1cm clip]{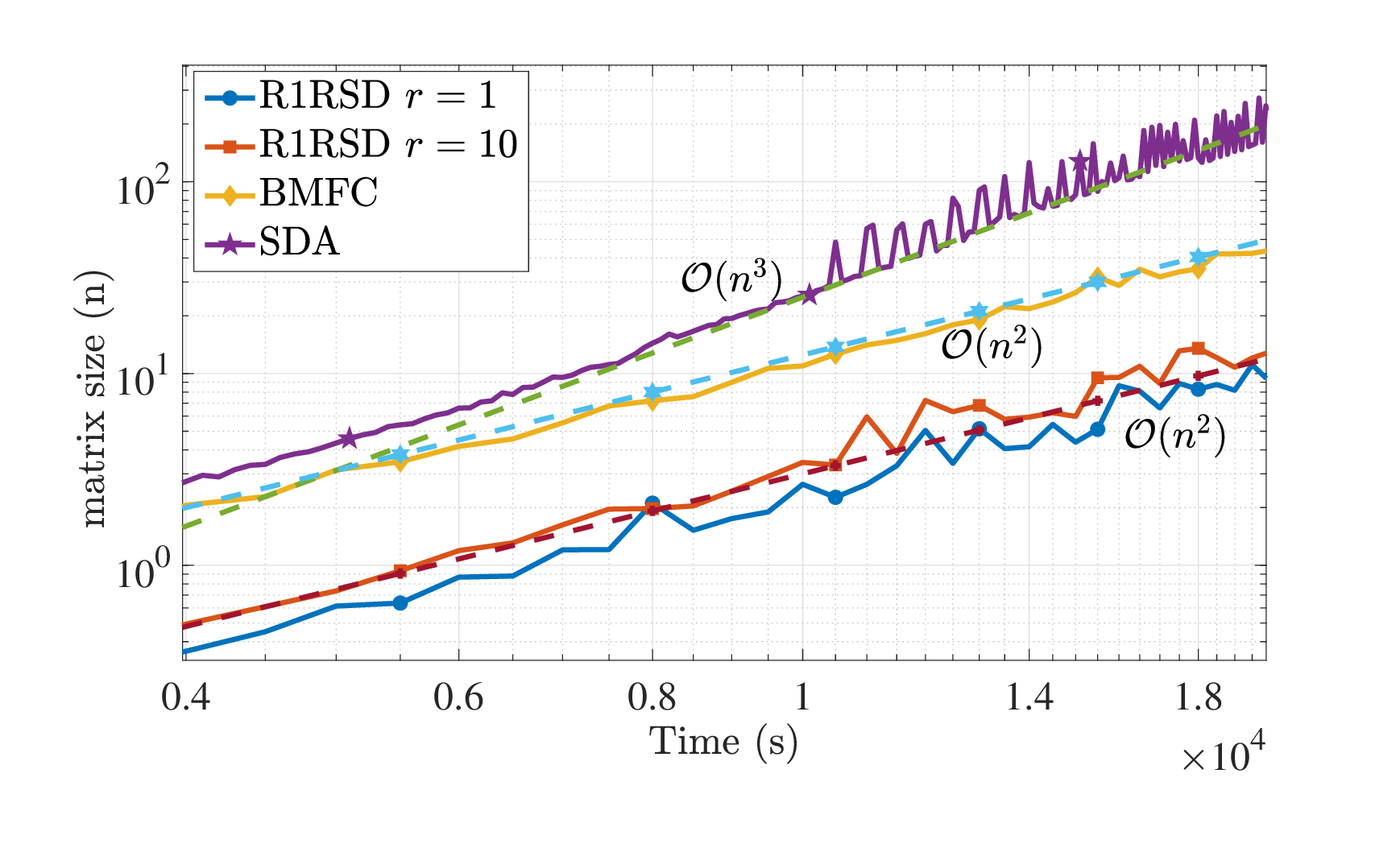}
\caption{Per-Iteration Complexity of R1RSD, BMFC, and SDA for DARE. The number of power iterations performed at each iteration is denoted by $r$.}
\label{per_iter_time_RSD_BM_SDA}
\end{figure} 
\subsubsection{BMFC coordinate selection}
Beyond computational cost, direction-selection strategies play a critical role in determining convergence behavior. For the BMFC method, we investigate both cyclic sampling and sampling-with-replacement strategies for coordinate-direction selection. In small-scale problems $n=100$, cyclic sampling consistently outperforms sampling with replacement in practice, likely due to its more uniform and systematic exploration of the coordinate directions. This empirical advantage is illustrated in Fig.~\ref{sampling_strat}. However, for large-scale problems, the two approaches exhibit comparable performance, suggesting that the benefits of cyclic ordering diminish as the problem dimension increases, as shown in the following subsections.
\begin{figure}[H]
	\includegraphics[width=0.98\columnwidth, trim = 1cm 1cm 8cm 1cm clip]{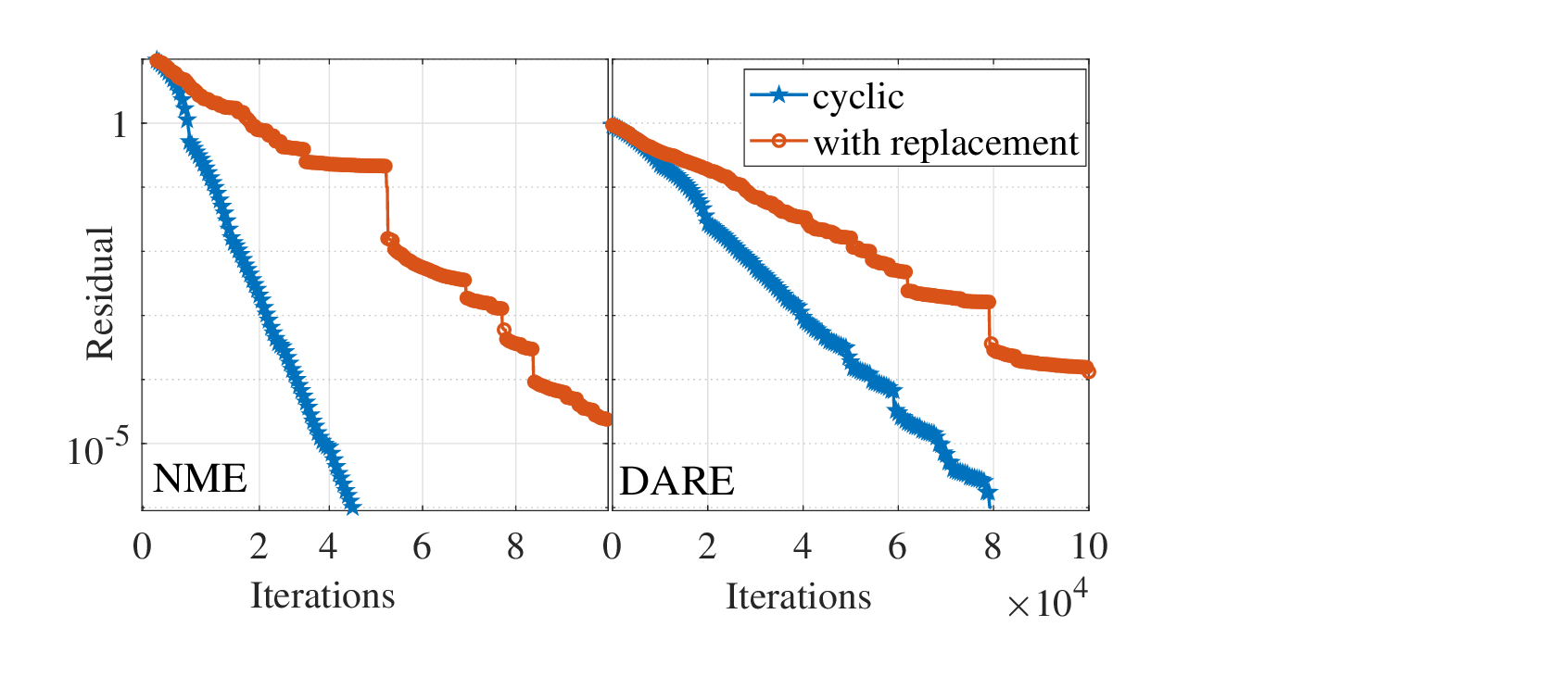}
	\caption{BMFC coordinate selection: cyclic vs sampling with replacement for problems of size $n = 100$.}
	\label{sampling_strat}
\end{figure}


\subsubsection{R1RSD direction selection}
For the proposed method, descent directions are computed using the power method. The number of power iterations serves as a tunable parameter that governs the trade-off between convergence quality and per-iteration computational cost. We interpret the single power-iteration variant as a randomized R1RSD method, while the greedy variant is obtained by running the power method for a sufficiently large number of iterations. Fig. \ref{rand_vs_greed} shows the performance of R1RSD for different number of power iterations. These results demonstrate that greedy subspace selection, corresponding to large $r$, substantially outperforms the randomized variant. Interestingly, using just 10 power iterations appears to provide an effective balance between performance and per-iteration cost. Further increasing the number of power iterations yields diminishing performance gains relative to the added computational burden. Notably, even the randomized variant of the proposed method outperforms the BMFC algorithm.

  \begin{figure}[H]
 	\includegraphics[width=\columnwidth, trim = 0cm 0cm 2.5cm 0cm, clip]{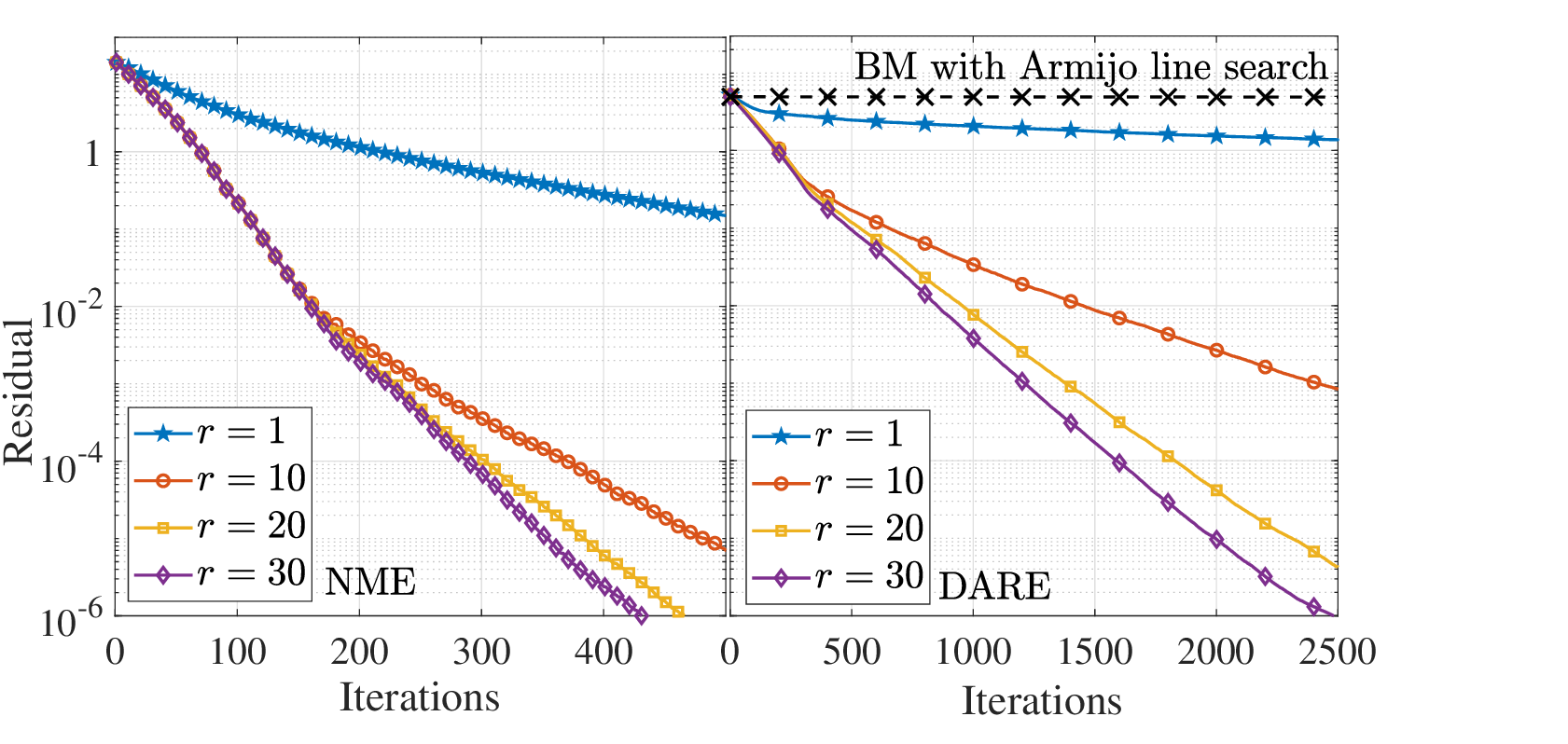}
 	\caption{Performance of R1RSD for different number of power iterations ($r$)}
 	\label{rand_vs_greed}
 \end{figure} 

\subsubsection{R1RSD step size}
Finally, we examine the effect of step-size selection on convergence. Although a tuned fixed step size can be employed, its performance deteriorates when the optimal step size varies significantly across iterations. This behavior is illustrated in Figs.~\ref{stepsize_adapt_vs_fixed} and \ref{greedy_vs_random_stepsize}, where NME~1 and NME~2 correspond to NMEs with different coefficient matrices  $\{\A,\Q\}$. While the fixed step size approximates exact line search reasonably well for NME~2, it performs poorly for NME~1 due to large variations in the optimal step-size $\beta$. In contrast, for the randomized variant of the proposed algorithm, the tuned fixed step size performs comparably to the line-search method.

\begin{figure}[H]
	\includegraphics[width=\columnwidth, trim = 2cm 4cm 3cm 1cm, clip]{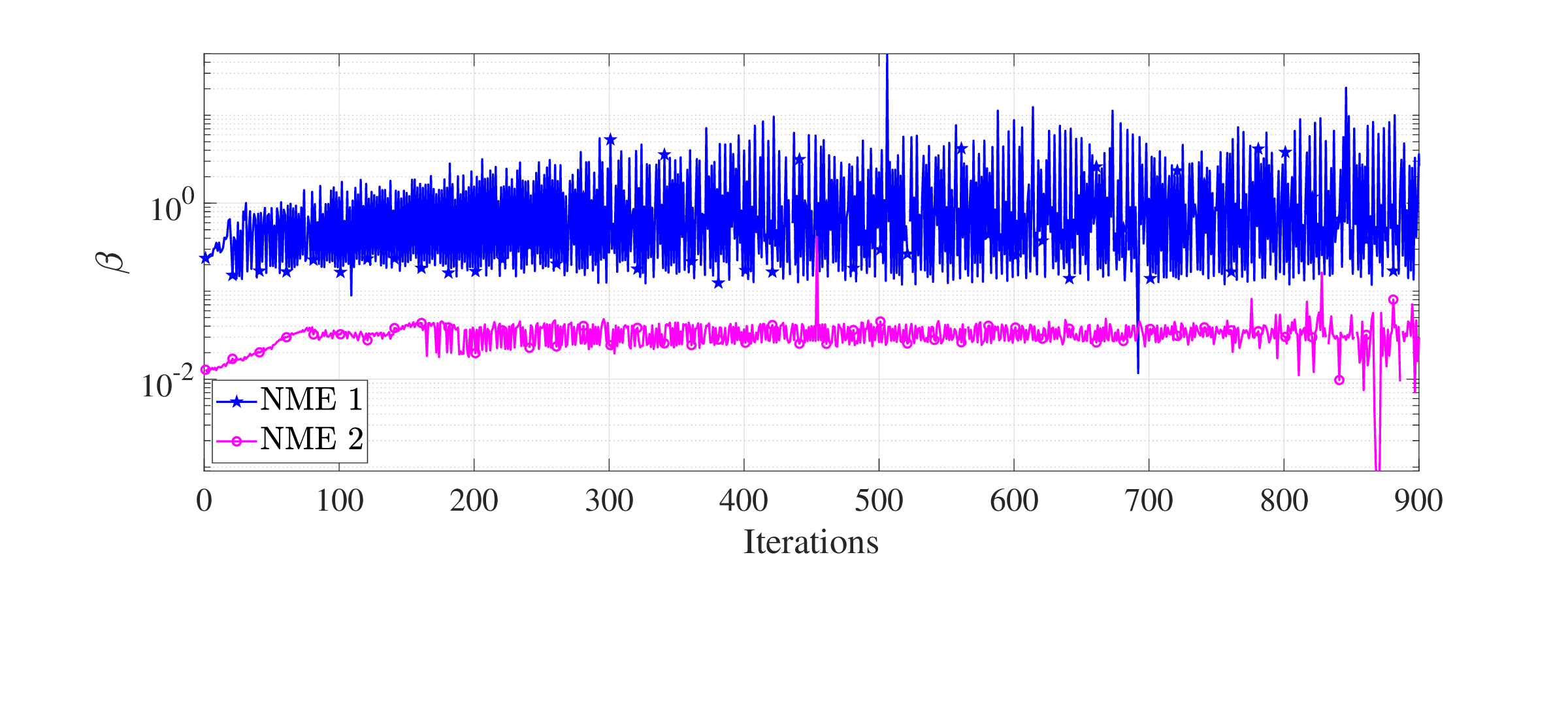}
	\caption{R1RSD step size $\beta$ selected via line search at each iteration for two instances of NME}
	\label{stepsize_adapt_vs_fixed}
\end{figure}

\begin{figure}[H]
	\includegraphics[width=\columnwidth, trim = 2.5cm 4cm 9cm 1cm, clip]{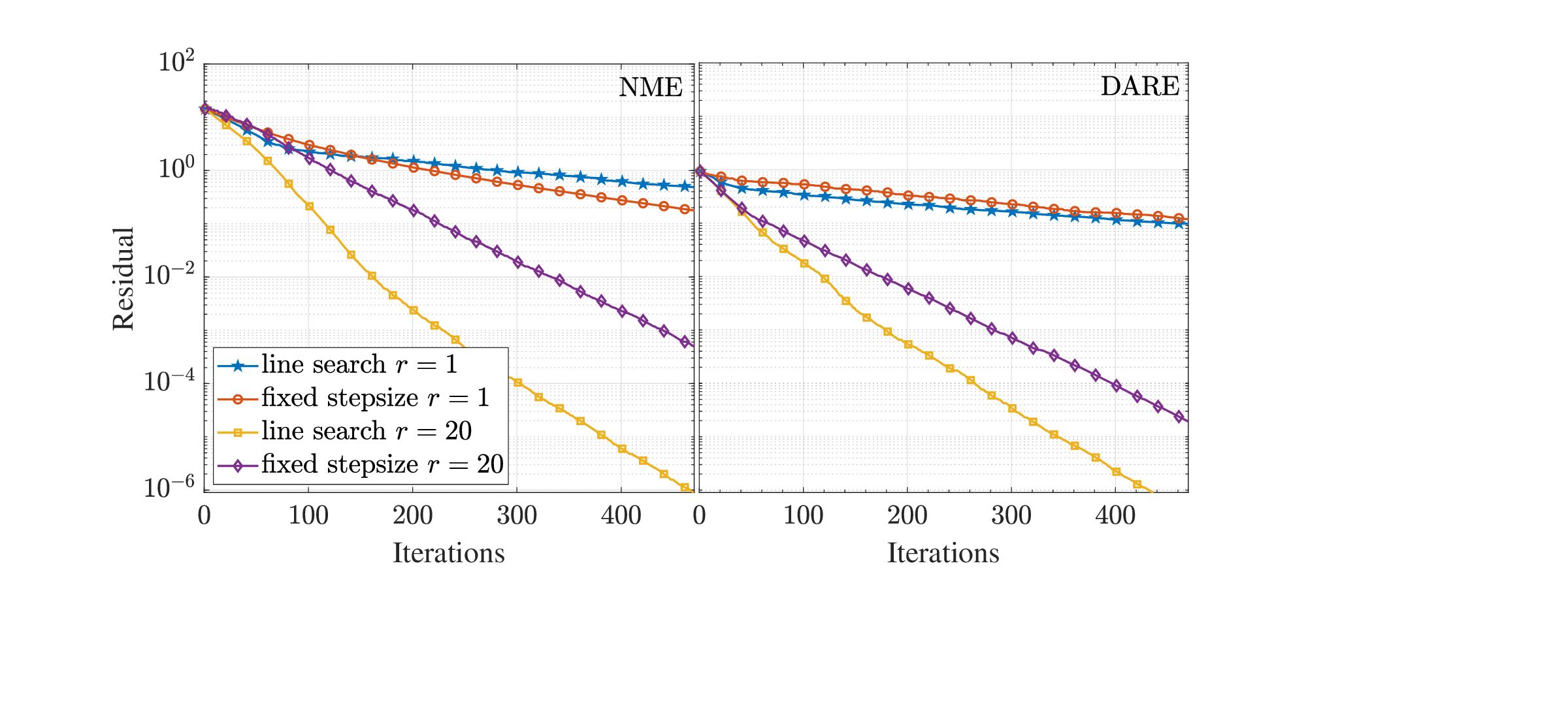}
	\caption{R1RSD step size selection: line search vs. fixed}
	\label{greedy_vs_random_stepsize}
\end{figure}

\begin{table}[t]
	\centering
	\setlength{\tabcolsep}{6pt}
	\renewcommand{\arraystretch}{1.05}
	\caption{Iteration complexities of R1RSD and BMFC for $n=100$: number of iterations required to ensure $f(\X) \leq \epsilon$}
	\label{small_table_result_compare}
	\begin{tabular}{llcc}
		\toprule
		Problem & Target & R1RSD (proposed) & BMFC \cite{han2024riemannian} \\
		\midrule
		\multirow{3}{*}{CARE} & $\epsilon = 10^{-2}$ &   698 & 50048 \\
		& $\epsilon = 10^{-4}$ &  1088 & 70712 \\
		& $\epsilon = 10^{-6}$ &  1470 & 91861 \\
		\midrule
		\multirow{3}{*}{DARE} & $\epsilon = 10^{-2}$ &   127 & 26599 \\
		& $\epsilon = 10^{-4}$ &   345 & 54062 \\
		& $\epsilon = 10^{-6}$ &   596 & 79395 \\
		\midrule
		\multirow{3}{*}{NME}  & $\epsilon = 10^{-2}$ &   165 & 20622 \\
		& $\epsilon = 10^{-4}$ &   390 & 33032 \\
		& $\epsilon = 10^{-6}$ &   666 & 45052 \\
		\bottomrule
	\end{tabular}
\end{table}

\subsubsection{Iteration complexity of small-scale problems} Table~\ref{small_table_result_compare} compares the iteration counts of R1RSD and BMFC for CARE, DARE, and NME at progressively tighter residual tolerances for problems with $n = 100$. Across all three problems and accuracy levels, the proposed R1RSD method consistently converges in one to two orders of magnitude fewer iterations than BMFC. These results highlight the effectiveness of the structured subspace selection and greedy descent strategy employed by R1RSD compared to the coordinate-based updates used in BMFC, even for small problem sizes.

 \subsection{Large scale Benchmarks}
 \subsubsection{Continuous-time algebraic Riccati equation (CARE)}  
 We begin by applying the proposed algorithm to the continuous-time algebraic Riccati equation (CARE), which serves as the first numerical example. The details of the exact line search method used for step size selection in both, the proposed and BFMC algorithms, is provided in Appendix \ref{stepcare}. We generate three instances of \eqref{care} for $n \in \{4000, 5000, 10000\}$ and solve them using the proposed and BMFC algorithms, SDA algorithm, and the MATLAB built-in function \texttt{icare}. When generating the coefficient matrices $\G$, $\A$, and $\H$, we ensured that their condition number was at most 15. The proposed and BMFC algorithms were initialized with the identity matrix. However, the SDA algorithm is sensitive to initialization and therefore it is initialized as specified in the literature \cite{huang2018structureI}. Though the power method is initialized randomly, the performance of the proposed algorithm did not show any variation across multiple runs. In each of the three instances, Table \ref{table_result_compare} shows the residual error, number of iterations, and the per-iteration wall clock time. For the \texttt{icare} algorithm, information about the number of iterations was not available, and only the total running time is reported. For the case of large $n$, both the algorithms were also run on a faster system so that they terminate within a reasonable amount of time. We make the following observation: (a) the convergence rate of the BMFC algorithm is slow, and the residual norm remains high even when the algorithm is run for several days. (b) SDA requires very few iterations but cannot handle $n \geq 5000$ on the machine with 8GB RAM (c) Similarly, \texttt{icare} cannot handle $n \geq 5000$, though its total time for $n = 4000$ is less than that of the proposed algorithm. (d) Unlike all state-of-the-art algorithms and standard solvers, the proposed R1RSD algorithm continues to work for problems with large $n$.

 \begin{table*}[t]
 	\centering
 	\setlength{\tabcolsep}{3pt}
 	\renewcommand{\arraystretch}{1.05}
 	\caption{Performance of various algorithms for solving CARE}
 	\label{table_result_compare}
 	\begin{tabular}{l ccc @{\hspace{8pt}} ccc @{\hspace{8pt}} ccc}
 		\toprule
 		& \multicolumn{3}{c}{$n=4000$}
 		& \multicolumn{3}{c}{$n=5000$}
 		& \multicolumn{3}{c}{$n=10000$} \\
 		\cmidrule(lr){2-4}\cmidrule(lr){5-7}\cmidrule(lr){8-10}
 		Algorithm
 		& residual & iters ($\times 10^5$) & time/iter (s)
 		& residual & iters ($\times 10^5$) & time/iter (s)
 		& residual & iters ($\times 10^5$) & time/iter (s) \\
 		\midrule
 		R1RSD (proposed)
 		& $<10^{-6}$ & 0.8 & 2.5
 		& $<10^{-6}$ & 1 & 3.6
 		& $<10^{-6}$ & 1.9 & 20.4 \\
 		BMFC (with replacement) \cite{han2024riemannian}
 		& $5.0\times 10^3$ & 2.4 & 21.5
 		& $6.2\times 10^3$ & 4.0 & 45.5
 		& $1.2\times 10^4$ & 8.0 & 356.3 \\
 		SDA \cite{huang2018structure}
 		& $<10^{-6}$ & -- & 116
 		& \multicolumn{3}{c}{out of memory}
 		& \multicolumn{3}{c}{out of memory} \\
 		\texttt{icare} (MATLAB)
 		& $<10^{-6}$ & -- & 31{,}204
 		& \multicolumn{3}{c}{out of memory}
 		& \multicolumn{3}{c}{out of memory} \\
 		\bottomrule
 	\end{tabular}
 \end{table*}
 
 \subsubsection{Nonlinear matrix equation}
 As in the earlier section, we compare the performance of the BFMC and the proposed algorithms for solving the NME \eqref{nme} of size $n = 5000$. The problem is created by generating the matrices $\A$ and $\Q$ randomly while ensuring that their condition number is at most 10. Both algorithms are implemented carefully using rank-one updates to ensure that the per-iteration complexity is $\O(n^2)$. Both the algorithms are initialized to $\X=\Q$. Step sizes are selected using the exact line search for both the algorithms and an efficient implementation of the process is detailed in Appendix \ref{nmestep}. As earlier, the SDA fails to work for this problem size on the machine with 8 GB RAM, while the proposed and BMFC algorithm continue to work.  Table~\ref{NME_result_compare} demonstrates the superior performance of the proposed algorithm in terms of both iteration count and total runtime. Interestingly, although the per-iteration computational cost of BMFC is lower than that of the R1RSD algorithm, the overall runtime of R1RSD is significantly smaller than that of BMFC due to its faster convergence. Moreover, the cyclic variant of BMFC exhibits performance comparable to that of the sampling-with-replacement variant.

\begin{table}[H]
	\centering
	\setlength{\tabcolsep}{3pt}
	\caption{Performance of R1RSD and BMFC for the NME}
	\label{NME_result_compare}
	\begin{tabular}{l c c c c}
		\toprule
		& \multicolumn{4}{c}{$n=5000$} \\
		\cmidrule(lr){2-5}
		Algorithm
		& Variant
		& residual 
		& iters ($\times 10^5$) 
		& time/iter (s) \\
		\midrule
		R1RSD (proposed)
		& -- & $<10^{-6}$ & .31 & 4.1 \\
		\multirow{2}{*}{BMFC \cite{han2024riemannian}}
		& cyclic & 714.05 & 5 & 1.7 \\
		& with replacement & 712.93 & 5 & 1.7 \\
		\bottomrule
	\end{tabular}
\end{table}

 	\subsubsection{Discrete-time algebraic Riccati equation (DARE)}
 	We consider solving \eqref{dare} of size $n = 5000$. The coefficient matrices are generated randomly to ensure that their condition numbers are below 15. As earlier, for the proposed algorithm, the step-size is selected via exact line search and the details are provided in Appendix \ref{stepdare}. For this problem however, exact line search for the BMFC algorithm cannot be realized within the $\O(n^2)$ per-iteration complexity budget and we instead use Amijo line search as also recommended in \cite{han2024riemannian}. Both the algorithms are initialized to $\X=\Q$.
 	Table \ref{dare_result_compare}  shows the performance of the proposed and BMFC algorithms. As earlier, we see that the proposed algorithm is superior in terms of both, the number of iterations required and the time required to ensure that the objective function is small.

 \begin{table}[H]
 	\centering
 	\setlength{\tabcolsep}{3pt}
 	\caption{Performance  of R1RSD and BMFC for the DARE}
 	\label{dare_result_compare}
 	\begin{tabular}{l c c c c}
 		\toprule
 		& \multicolumn{4}{c}{$n=5000$} \\
 		\cmidrule(lr){2-5}
 		Algorithm
 		& Variant
 		& residual 
 		& iters ($\times 10^4$) 
 		& time/iter (s) \\
 		\midrule
 		R1RSD (proposed)
 		& -- & $<10^{-5}$ & 3 & 4.8 \\
 		\multirow{2}{*}{BMFC \cite{han2024riemannian}}
 		& cyclic & 47.05 & 5.6 & 15.4  \\
 		& with replacement & 47.04 & 7.15 & 15.5 \\
 		\bottomrule
 	\end{tabular}
 \end{table}

\section{Conclusion}\label{sec:conclusion}
In this work, we studied the computation of symmetric positive definition (SPD) solution of nonlinear matrix equations by recasting them as residual norm minimization problems on the SPD manifold. We proposed a rank-one Riemannian subspace descent method that updates the iterate along a dominant eigen-component of a transformed Riemannian gradient, identified via a small number of power iterations. The proposed updates admits exact step-size selection for broad classes of objectives, while still incurring only $\O(n^2\log(n))$ cost per-iteration and requiring at most $\O(n)$ iterations. Numerical experiments on large-scale CARE, DARE, and additional nonlinear matrix equations support the analysis and show that the method remains effective in dense regimes where cubic-cost iterations become impractical. In particular, the proposed algorithm solves instances up to $n=10{,}000$ in our tests for which the compared solvers, including MATLAB's \texttt{icare}, structure-preserving doubling algorithms, and subspace-descent baselines, do not return a solution under the same computational budget. These results indicate that rank-one manifold updates can serve as a viable alternative for high-dimensional SPD-constrained matrix equations found in control theory and related areas. 

\appendices

\section{Updates for CARE \eqref{care}}\label{stepcare}
Let us define 
\begin{align}
	\M(\X) &:= \X\G\X - \A^\T\X - \X\A - \H \\
	\N(\X) &:= (\G\X - \A)\M(\X)
\end{align}
where we assume, as is generally the case with CARE, that $\G$ and $\H$ are symmetric. The residual norm function and its gradient becomes:
\begin{align}
	f(\X) &= \norm{\M(\X)}_F^2 \\
	\grd f(\X) &= \N(\X) + \N^\T(\X)
\end{align}
For ease of implementation, we will maintain $\M_t = \M(\X_t)$ and $\N_t = \N(\X_t)$, in addition to $\X_t$ and $\B_t$ for all $t$. Let us define a few intermediate quantities, all of which can be calculated in $\O(n^2)$: 
\begin{align}
	\v_t &= \B_t\y_t & \w_t &= (\X\G - \A^\T)\v_t & \omega_t &= \v_t^\T\G\v_t.
\end{align}
Then the update can be written as $\X_{t+1} = \X_t + \alpha_t \v_t\v_t^\T$. We can maintain $\M_t$ efficiently as
\begin{align}
	\M_{t+1} &= \M_t + \alpha_t (\v_t\w_t^\T + \w_t\v_t^\T) + \alpha_t^2 \omega_t\v_t\v_t^\T\label{mtup}
\end{align}
For the line search, we observe that
\begin{align}
	&f(\X_{t+1}) = \tr{\M_{t+1}^2} = f(\X_t) + 4\alpha_t\v_t^\T\M_t\w_t \nonumber\\
	& + 2\alpha_t^2\left(\omega_1\v_t^\T\M_t\v_t + (\v_t^\T\w_t)^2 + \norm{\v_t}_2^2\norm{\w_t}_2^2\right)\\
	& + 4\alpha_t^3 \omega_t(\v_t^\T\w_t)\norm{\v_t}_2^2 + \alpha_t^4\omega_t^2\norm{\v_t}_2^4
\end{align}
which is a quartic equation in $\alpha_t$ and can be easily minimized with respect to $\alpha_t$. We used MATLAB in-built function \texttt{fminbnd} to search for the optimal value in the range $(-1,10)$. Likewise, writing $\N_{t+1} = (\G\X_{t+1}-\A)\M_{t+1}$ and substituting the updates for $\X_{t+1}$ and $\M_{t+1}$, we obtain the $\O(n^2)$ update for $\N_{t+1}$:
\begin{align}
	\N_{t+1}
	&= \N_t
	+ \alpha_t(\G\X_t-\A)\Big(\v_t\w_t^\T+\w_t\v_t^\T+\alpha_t\omega_t\,\v_t\v_t^\T\Big) \nonumber\\
	&+ \alpha_t(\G\v_t)\Big((\M_t\v_t)^\T
	+ \alpha_t\|\v_t\|_2^2\w_t^\T
	+ \alpha_t(\v_t^\T\w_t)\v_t^\T\Big) \nonumber\\
	&+  \alpha_t^3\omega_t\|\v_t\|_2^2(\G\v_t)\v_t^\T
	\label{Ntup}
\end{align}
obviating the need to calculate the gradient at every iteration. 

The BMFC algorithm uses Burer-Monteiro factorization $\X=\Y\Y^\T$. The update equation has the following rank-one update form $\Y_{t+1}=\Y_t-\delta_t\e_i\e_j^T$. In this case, $\X_t$ can be maintained via a rank-two update as 
\begin{align}
	\X_{t+1} = \X_t + \U_t\Delta_t\U_t^\T \label{Xtup}
\end{align}
where $\U_t = \mat{\Y_t\e_j & \e_i}$ and $\Delta_t = \mat{0 & -\delta_t \\ -\delta_t & \delta_t^2}$. We can write  a similar intermediate variable 
\begin{align}
	\M'(\Y) &= \Y\Y^\T\G\Y\Y^\T - \A^\T\Y\Y^\T - \Y\Y^\T\A - \H
\end{align}
and maintain $\M'_t = \M'(\Y_t)$ for all $t$. Since $\M'(\Y)$ is quartic in $\Y$, a rank-one update to $\Y$ results in a rank four update to $\M'(\Y)$. Defining 
\begin{align}
	\P_t &= (\X_t\G - \A^\T)\U_t\Delta_t & \Delta'_t = \Delta_t \U_t^\T\G\U_t\Delta_t
\end{align} 
Then the update for $\M_{t+1}$ can be written as 
\begin{align}	\label{Mtprime_up}
	&\M'_{t+1} = \M'_t + \P_t\U_t^\T  + \U_t\P_t^\T + \U_t\Delta'_t\U_t^\T\nonumber
\end{align}
Here $\P_t$ and $\U_t$ are $n \times 2$ matrices, so each term can be calculated in $\O(n^2)$ time. Finally, we can update the function value as 
\begin{align}
	&f(\X_{t+1})
	= f(\X_t) +  4\tr{\U_t^\T\M'_t\P_t}  + 2\tr{\U_t^\T\M'_t\U_t\Delta'_t} \nonumber\\
	&+\tr{\Delta'_t(\U_t^\T \U_t)\Delta'_t(\U_t^\T \U_t)\Big)}+2\tr{(\P_t^\T \P_t)(\U_t^\T \U_t)}\nonumber\\
	&+2\tr{(\U_t^\T \P_t)(\U_t^\T \P_t)} +2\tr{\Delta'_t(\U_t^\T \U_t)(\P_t^\T \U_t)}\nonumber\\
	&+2\tr{\Delta'_t(\U_t^\T \P_t)(\U_t^\T \U_t)}.
\end{align}
where all the traces involve $2 \times 2$ matrices and hence the update can be carried out in $\O(n^2)$ time. Additionally, it can be seen that $f(\X_{t+1})$ is quadratic in $\Delta'_t$, quartic in $\Delta_t$ and hence a degree-8 polynomial in $\delta_t$. We used MATLAB built-in function \texttt{fminunc} to minimize the objective. 

\section{Updates for NME \eqref{nme}}\label{nmestep}
For the NME, we assume that $\Q$ is symmetric. Defining 
\begin{align}
	\M(\X) &= \X+ \A^\T\X^{-1}\A - \Q \\
	\N(\X) & = \X^{-1}\A\M(\X)\A^\T\X^{-1}
\end{align}
we have that 
\begin{align}
	f(\X) &= \norm{\M(\X)}_F^2 & \grd f(\X) &= 2\M(\X) - 2\N(\X).
\end{align}
For the update $\X_{t+1} = \X_t + \alpha_t \v_t\v_t^\T$ where $\v_t = \B_t\y_t$, we can efficiently maintain $\X_{t+1}^{-1} = \X_t^{-1} + \mu_t \z_t\z_t^\T$ where $\mu_t = -\frac{\alpha_t}{1+\alpha_t}$ and $\z_t = \B_t^{-T}\y_t$, which can be calculated efficiently by back-substitution. Also defining $\w_t = \A^\T\z_t$, we can maintain $\M_t = \M(\X_t)$ for all $t$ as 
\begin{align}
	\M_{t+1} = \M_t  + \alpha_t \v_t\v_t^\T + \mu_t \w_t\w_t^\T
\end{align}
Likewise, the function value can be updated as
\begin{align}
	f(\X_{t+1}) &= f(\X_t) + 2\alpha_t\v_t^\T \M_t \v_t
	+2\mu_t\w_t^\T\M_t \w_t \nonumber\\
	&+\alpha_t^2\|\v_t\|_2^4+\mu_t^2\|\w_t\|_2^4 +2\alpha_t\mu_t\,(\v_t^\T\w_t)^2 .
\end{align}
which is a rational expression in $\alpha_t$ and must be minimized for exact line search. Finally to maintain the gradient, we need to maintain $\N_t = \N(\X_t)$ for all $t$. For this define
\begin{align*}
	\g_t &:= \X_t^{-1}\A\v_t, &
	\h_t &:= \X_t^{-1}\A\w_t, &
	\s_t &:= \X_t^{-1}\A(\M_{t+1}\w_t), 
\end{align*}
so that the update can be written as
\begin{align}
	\N_{t+1} &= \N_t  + \alpha_t \g_t \g_t^\T + \mu_t \h_t \h_t^\T + \mu_t\big(\z_t \s_t^\T + \s_t \z_t^\T\big) \nonumber\\
	& + \mu_t^2(\w_t^\T \M_{t+1}\w_t) \z_t\z_t^\T .
	\label{Nt_NME_up_compact}
\end{align}

In the BMFC algoithm, since we have the update $\Y_{t+1}=\Y_t-\delta_t\e_i\e_j^T$, the inverse of $\Y_t$ can be maintained as $\Y_{t+1}^{-1} = \Y_t^{-1} + \phi_t\p_t\q_t^\T$ where $\p_t = \Y_t^{-1}\e_i$, $\q_t = \Y_t^{-1}\e_j$, and $\phi_t = \delta_t/(1 -\delta_t \e_i^\T\Y_t^{-1}\e_j)$. We can maintain $\X_{t+1}$ using a rank-two update as in \eqref{Xtup}. In order to maintain $\X_{t+1}^{-1}$, we define $\V_t = \mat{\Y_t^{-T}\p_t & \q_t}$, $\varphi_t = \p_t^\T\p_t$, and
\begin{align}
	\Phi_t = \mat{0 & \phi_t \\ \phi_t & \phi_t^2\varphi_t}
\end{align}
so that we have the rank-two update:
\begin{align}
	\X_{t+1}^{-1} = \X_t^{-1} + \V_t \Phi_t\V_t^\T.
\end{align}
Finally, for the function value update, define $\W_t = \A^\T\V_t$ and $\Delta'_t = \U_t\Delta_t\U_t^\T+\W_t\Phi_t\W_t^\T$. Then the function value update is given by 
\begin{align}
	&f(\X_{t+1})= f(\X_t) +2\tr{\Delta_t\U_t^\T \M_t \U_t} \nonumber\\
	& +2\tr{\Phi_t\W_t^\T \M_t \W_t}+\tr{\Delta_t(\U_t^\T \U_t)\Delta_t(\U_t^\T \U_t)}\nonumber\\
	& +\tr{\Phi_t(\W_t^\T \W_t)\Phi_t(\W_t^\T \W_t)} \nonumber\\
	&+2\tr{\Delta_t(\U_t^\T \W_t)\Phi_t(\U_t^\T \W_t)^\T}
	\label{ft_BMFC_NME}
\end{align}
where all trace operations involve $2 \times 2$ matrices. This is a rational expression in $\delta_t$ and can be numerically minimized with respect to $\delta_t$ for exact line search. 

\section{Updates for DARE \eqref{dare}}\label{stepdare}
We define intermediate matrix variable $\M_{1}(\X)=(\I+\G\X)^{-1}$, so that the residual matrix $\M(\X)$ can written as
\begin{align}
	\M(\X)&=\X-\A^\T\X\M_1(\X)\A-\Q.
\end{align}
We also introduce the following intermediate variables 
\begin{align*}
	\M_2(\X) &= \A\M(\X)\A^\T\M_{1}^\T(\X), &\M_{3}(\X)=  \G^\T\M_1(\X)^\T\X
\end{align*}
so that the gradient of the residual norm $f(\X) = \norm{\M(\X)}_F^2 $ can written as
\begin{align}
	\grd \; f(\X)&=\left[\M(\X)^{\T}+\M_{2}(\X)^\T(-\I+\M_{3}(\X)^\T) \right]\nonumber\\
	& +\left[\M(\X)+(-\I+\M_{3}(\X))\M_{2}(\X) \right]
\end{align}
Observe now that the intermediate variables $\M_t = \M(\X_t)$, $\M_{1,t}=\M_{1}(\X_t)$,  $\M_{2,t}=\M_{2}(\X_t)$ and $\M_{3,t}=\M_{3}(\X_t)$ can be recursively updated. To this end, let us define the quantities
\begin{align}
	\u_t&:=\B_t\y_t & \s_t&=\A^\T\u_t, & \v_t^\T&=\u_t^{\T}\M_{1,t}\A,
\end{align}
\begin{align}
	\h_t^\T &=  \u_t^{\T}\G\M_{1,t}^\T, & \w_t&= \A^\T\X_t\h_t, \\
	a_1&=\u_t^{\T}\A^\T\X_t\M_{1,t}\G\u_t & a_2&=\u_t^\T\h_t,\\
	a_3&=\u_t^{\T}\v_t,&a_4&= \v_t^{\T}\v_t, \\
	b_1&=(1+\alpha_{t}a_2)^{-1} & \r_t& = \s_t - \w_t.
\end{align}
For the update $\X_{t+1} = \X_t + \alpha_t \u_t\u_t^\T$, we can calculate $\M_{1,t+1}=(\I+\G\X_{t+1})^{-1}$ using the recursion 
\begin{align}
	\M_{1,t+1} &=\M_{1,t}-\alpha_tb_1\h_t\u_t^{\T}\M_{1,t}.
\end{align}
To calculate the update for $\M_{t+1} = \X_{t+1} - \A^\T\X_{t+1}\M_{1,t+1}\A-\Q$, observe that
\begin{align*}
	\A^\T\X_{t+1}\M_{1,t+1}\A &= \A^\T\X_t\M_{1,t}\A + \alpha_tb_1\r_t\v_t^\T \\
	\Rightarrow \M_{t+1} &= \M_t + \alpha_t\u_t\u_t^\T - \alpha_tb_1\r_t\v_t^\T
\end{align*}
Hence the function value update takes the form
\begin{align}
	&f(\X_{t+1}) = f(\X_t) +2\alpha_t\,\u_t^\T \M_t \u_t -2\alpha_tb_1\v_t^\T \M_t^\T \r_t \nonumber\\
	&+\alpha_t^2\norm{\u_t}^4 -2\alpha_t^2b_1a_3(\u_t^\T \r_t) +a_4\alpha_t^2b_1^2\norm{\r_t}^2.
\end{align}
which can be calculated in $\O(n^2)$ time and can be minimized with respect to $\alpha_t$ for optimal step-size selection.  Further, we can calculate
\begin{align}
	\M_{2,t+1}^\T\v = \M_{1,t+1}\A\M_{t+1}^\T\A\v \\
	\M_{3,t+1}^\T\u_t = \X_{t+1}\M_{1,t+1}^\T\G\u_t 
\end{align}
so that $\B_t^\T \grd f(\X_{t+1}) \B_t \y_t$ can also be calculated in $\O(n^2)$. 

For the BMFC algorithm, the expression of the residual error is highly nonlinear in the step size $\delta_t$, making the direct computation of an optimal $\delta_t$ impractical within reasonable time. Consequently, we adopt the Armijo line search strategy for step-size selection. This approach requires evaluating the residual matrix 
\begin{align}
	\M(\Y)&=\Y\Y^\T-\A^\T\Y\Y^\T\left(\I+\G\Y\Y^\T\right)^{-1}\A-\Q
\end{align} 
in $\O(n^2)$ time by maintaining intermediate variables. 
Let $\n_1 := \Y_t\e_j$, $\n_2 := \A^\T\n_1$, and $\n_3 := \A^\T\e_i$. Since the update $\Y_{t+1}=\Y_t-\delta_t\e_i\e_j^\T$ induces a rank-one modification, the update of $\X_t$ can be expressed as
\begin{align}
	\X_{t+1}
	=\X_t+\sum_{k=1}^3 \u_{x,k}\v_{x,k}^\T, 	\label{up_BM_X}
\end{align}
where
\begin{align}
	(\u_{x,1},\v_{x,1})&=(-\delta_t\e_i,\n_1),\;
	(\u_{x,2},\v_{x,2})=(-\delta_t\n_1,\e_i),\nonumber\\
	(\u_{x,3},\v_{x,3})&=(\delta_t^2\e_i,\e_i)\nonumber.
\end{align}

Define $\N(\Y)=\I+\G\Y\Y^\T$. The intermediate matrix $\N_t=\N(\Y_t)$ can be updated recursively as
\begin{equation}
	\label{app:N}
	\N_{t+1}
	=\N_t+\sum_{k=1}^3 \u_{n,k}^t(\v_{n,k}^t)^\T,
\end{equation}
with
\begin{align}
	(\u_{n,1}(\Y),\v_{n,1}(\Y))&=(-\delta_t\G\e_i,\Y\e_j),\nonumber\\
	(\u_{n,2}(\Y),\v_{n,2}(\Y))&=(-\delta_t\G\Y\e_j,\e_i),\nonumber\\
	(\u_{n,3}(\Y),\v_{n,3}(\Y))&=(\delta_t^2\G\e_i,\e_i).
\end{align}
and $\u_{n,i}^t=\u_{n,i}(\Y_t)$, $\v_{n,i}^t=\v_{n,i}(\Y_t)$.

The inverse $\N_{t+1}^{-1}$ can be updated in $\mathcal{O}(n^2)$ time by applying successive Sherman–Morrison updates. Specifically, we define
\begin{align}
	\N_{2,t} &= \N_t + \u_{n,1}^t (\v_{n,1}^t)^\T,\
	\N_{1,t} &= \N_{2,t} + \u_{n,2}^t (\v_{n,2}^t)^\T,
\end{align}
and compute their inverses recursively, leading finally to
\begin{align}
	\N_{t+1} = \N_{1,t} + \u_{n,3}^t (\v_{n,3}^t)^\T,
\end{align}
whose inverse is obtained with another rank-one update.

Using these intermediate quantities, the residual matrix
\begin{align}
	\M(\Y)=\X-\A^\T\X\N(\Y)^{-1}\A-\Q
	\label{res_mat}
\end{align}
can be evaluated efficiently. In particular, the term $\A^\T\X\N(\Y)^{-1}\A$ is maintained through the intermediate matrices $\A^\T\X_t\N_{2,t}^{-1}\A$ and $\A^\T\X_t\N_{1,t}^{-1}\A$, which are updated via rank-one corrections. Specifically,
\begin{align}
	\A^\T\X_t\N_{1,t}^{-1}\A
	&= \A^\T\X_t\N_{2,t}^{-1}\A
	+ \frac{\delta_t\p_{1,t}\q_{1,t}^{\T}}{1+ \delta_t m_{1,t}},
\end{align}
where $\p_{1,t}=\A^\T\X_t\N_{2,t}^{-1}\G\Y_t\e_j$, $\q_{1,t}=\A\N_{2,t}^{-\T}\e_i$, and $m_{1,t}=\e_i^\T\N_{2,t}^{-1}\G\Y_t\e_j$, and
\begin{align}
	\A^\T\X_t \N_{2,t}^{-1}\A
	&= \A^\T\X_t\N_t^{-1}\A
	+ \frac{\delta_t\p_{2,t}\q_{2,t}^\T}{1+ \delta_t m_{2,t}},
\end{align}
where $\p_{2,t}=\A^\T\X_t\N_t^{-1}\G\e_i$, $\q_{2,t}=\A\N_{t}^{-\T}\Y_t\e_j$, and $m_{2,t}=\e_j^\T\Y_t^\T\N_t^{-1}\G\e_i$. These relations together yield
\begin{align}
	\A^\T\X_{t+1}\N_{t+1}^{-1}\A
	&= \A^\T\X_t\N_{1,t}^{-1}\A
	-\delta_t\n_3\s_t^\T
	+\r_t\q_t^\T, \label{up_BM_intmat}
\end{align}
with
\begin{align}
	\p_t&=\A^\T\X_t\N_{1,t}^{-1}\G\e_i,\quad
	\q_t= \A\N_{1,t}^{-\T}\e_i,\quad
	\s_t= \A\N_{1,t}^{-\T}\n_1,\\
	\r_t &= \left(\delta_t^2 - \frac{c_4\delta_t^4}{c_1} + \frac{c_2\delta_t^3}{c_1}\right)\n_3
	- \frac{\delta_t^2\p_t}{c_1}
	+ \left(-\delta_t + \frac{c_3\delta_t^3}{c_1}\right)\n_2,\\
	c_1 &= 1+ \delta_t^2\e_i^\T \N_{1,t}^{-1} \G\e_i,\quad
	c_2= \n_1^\T\N_{1,t}^{-1}\G\e_i,\\
	c_3 &= \e_i^\T\N_{1,t}^{-1}\G\e_i,\quad
	c_4= \e_i^\T\N_{1,t}^{-1}\G\e_i.
\end{align}
Consequently, by combining \eqref{up_BM_X} and \eqref{up_BM_intmat}, the residual matrix \eqref{res_mat} can be evaluated in $\mathcal{O}(n^2)$ time, thereby enabling the Armijo line search to be performed with $\mathcal{O}(n^2)$ computational complexity.

\footnotesize

\bibliographystyle{IEEEtran} 
\bibliography{IEEEabrv,maniopt_ref}

\end{document}